\documentclass[10pt]{article}
\usepackage{amsmath,amssymb,graphicx,float,url,enumitem,xcolor,color,soul,subcaption}
\usepackage{placeins}
\usepackage{amsthm}
\usepackage[colorlinks=true]{hyperref}

\newcommand{\norm}[1]{\left\lVert #1 \right\rVert}
\newcommand{\inner}[2]{\left\langle #1,#2 \right\rangle}
\newcommand{\abs}[1]{\left|#1\right|}
\newcommand{\un}[0]{\boldsymbol{\nu}}

\renewcommand{\div}[0]{\textnormal{div}\,}
\newcommand{\sdiv}[0]{\textnormal{div}}
\newcommand{\curl}[0]{\textnormal{\textbf{curl}}\,}
\newcommand{\scurl}[0]{\textnormal{curl}}

\newcommand{\Hcurl}[1]{\mathbf{H}(\textnormal{\textbf{curl}},#1)}
\newcommand{\Hdiv}[1]{\mathbf{H}(\textnormal{div},#1)}
\newcommand{\inncurl}[3]{(#1,#2)_{\textnormal{\textbf{curl}},#3}}
\newcommand{\normcurl}[2]{\norm{#1}_{\textnormal{\textbf{curl}},#2}}

\newcommand{\Hcurlloc}[1]{\textbf{H}_{\textnormal{loc}}(\textbf{curl},#1)}
\newcommand{\veccurl}[0]{\mathbf{curl}}
\newcommand{\boldphi}[0]{\boldsymbol{\varphi}}
\newcommand{\boldPsi}[0]{\boldsymbol{\Psi}}
\newcommand{\boldH}[0]{\boldsymbol{\mathcal{H}}}

\newcommand{\supp}{\textnormal{supp}}
\newcommand*\myat{{\fontfamily{ptm}\selectfont @}}
\renewcommand{\Im}[0]{\textnormal{Im}}
\renewcommand{\Re}[0]{\textnormal{Re}}

\usepackage{empheq}
\usepackage[many]{tcolorbox}

\tcbset{
  highlight math style={
    colback=yellow,
    arc=0pt,
    outer arc=0pt,
    boxrule=0pt,
    top=2pt,
    bottom=2pt,
    left=2pt,
    right=2pt,
  }
}

\renewenvironment{proof}{{\bfseries Proof.}}{\qed} 
\newtheorem{theorem}{Theorem}
\newtheorem{lemma}[theorem]{Lemma}
\newtheorem{definition}[theorem]{Definition}
\newtheorem{corollary}[theorem]{Corollary}

\newtheorem{remark}[theorem]{Remark}

\newtheorem{assumption}[theorem]{Assumption}
\newtheorem{proposition}[theorem]{Proposition}

\usepackage{titlesec}
\titleformat{\section}
  {\normalfont\fontsize{10}{15}\bfseries}{\thesection.\!\!\!\!}{1em}{}
\titleformat{\subsection}
  {\normalfont\fontsize{10}{15}\bfseries}{\thesubsection.\!\!\!\!}{1em}{}

\numberwithin{theorem}{section}
\allowdisplaybreaks

\title{Existence and stability of electromagnetic Stekloff eigenvalues with a trace class modification}
\author{Samuel Cogar\thanks{Department of Mathematics, Rutgers University, Piscataway, NJ 08854 (samuel.cogar\myat rutgers.edu}}
\date{}

\begin{document}

\maketitle

\begin{abstract}

A recent area of interest is the development and study of eigenvalue problems arising in scattering theory that may provide potential target signatures for use in nondestructive testing of materials. We consider a generalization of the electromagnetic Stekloff eigenvalue problem that depends upon a smoothing parameter, for which we establish two main results that were previously unavailable for this type of eigenvalue problem. First, we use the theory of trace class operators to prove that infinitely many eigenvalues exist for a sufficiently high degree of smoothing, even for an absorbing medium. Second, we leverage regularity results for Maxwell's equations in order to establish stability results for the eigenvalues with respect to the material coefficients, and we show that this generalized class of Stekloff eigenvalues converges to the standard class as the smoothing parameter approaches zero.

\end{abstract}



\section{Introduction} \label{sec_intro}

It is important in many areas of science and engineering to be able to determine whether a given material is defective without compromising the integrity of the material in the process. Such techniques for nondestructive evaluation often involve interrogating the medium with a prescribed acoustic, elastic, or electromagnetic incident wave and observing the resulting scattering effects, and the resulting data is used to deduce information about the medium such as its support, connectivity, and consitutive parameters. Determination of the latter property brings many interesting difficulties; in particular, an anisotropic medium may not be uniquely determined by the measured scattering data (cf. \cite{colwell}), leading to some anbiguity when using iterative methods to compute an approximation to the constitutive parameters. \par

However, in nondestructive evaluation it is not necessary to determine the constitutive parameters of a potentially damaged sample, as the only information sought is that a significant change has occurred in these parameters. This notion brings us to our current line of investigation, in which we seek to develop target signatures that carry information about a medium and whose observed shifts allow us to infer changes in the constitutive parameters of a sample relative to a reference configuration. A recent problem of interest is to study eigenvalues arising from scattering theory as potential target signatures, with the theory of transmission eigenvalues serving as an early example. We refer to \cite{cakoni_colton_haddar} for a comprehensive survey of this theory. In order to overcome some practical difficulties with the potential use of transmission eigenvalues in this manner (cf. \cite{cogar}), a collection of new eigenvalue problems has been generated by comparing the measured scattering data to that of an auxiliary scattering problem that is independent of the medium under investigation (cf. \cite{audibert_cakoni_haddar,audibert_chesnel_haddar2018,audibert_chesnel_haddar2019,bi_zhang_yang,cakoni_colton_meng_monk,cogar,cogar2019,cogar2020,cogar_colton_meng_monk,cogar_colton_monk,cogar_colton_monk2019,cogar_monk,liu_liu_sun}). In the present study we adopt this strategy in the context of electromagnetic scattering theory. \par

We concern ourselves with a modification of the first such eigenvalue problem studied (cf. \cite{camano_lackner_monk}), which was generated by choosing the auxiliary problem corresponding to electromagnetic scattering by an impenetrable obstacle with an impedance condition enforced on its boundary. The initial attempt at this problem was a direct generalization of Stekloff eigenvalues first considered for acoustic scattering in \cite{cakoni_colton_meng_monk}, but the more strict compactness requirements associated with the analysis of Maxwell's equations led the authors to consider a slightly modified problem in which a projection operator $\mathcal{S}_0$ was introduced into the boundary condition. They considered the eigenvalue problem
\begin{subequations} \label{stek}
\begin{align} 
\curl\curl\mathbf{w} - k^2\epsilon\mathbf{w} &= \mathbf{0} \text{ in } B, \label{stek1} \\
\un\times\curl\mathbf{w} - \lambda\mathcal{S} \mathbf{w}_T &= \mathbf{0} \text{ on } \partial B, \label{stek2}
\end{align}
\end{subequations}
where $\lambda$ is the eigenparameter and $\mathcal{S}$ may represent either the identity $I$ or the aforementioned projection operator $\mathcal{S}_0$. In both cases the values of $\lambda$ for which nontrivial solutions $\mathbf{w}$ exist were referred to as \emph{electromagnetic Stekloff eigenvalues}, but we will refer to the eigenvalues corresponding to $\mathcal{S} = \mathcal{S}_0$ as the \emph{standard electromagnetic Stekloff eigenvalues} due to the close relationship of this problem and the one we will introduce shortly. \par

It was shown in \cite{camano_lackner_monk} that the eigenvalues of \eqref{stek} with $\mathcal{S} = \mathcal{S}_0$ form an infinite discrete set without finite accumulation point when the coefficient $\epsilon$ is real-valued, and we also remark that this result was extended to both versions of \eqref{stek} in \cite{halla2,halla1}. However, in our intended application of nondestructive evaluation, many materials have a significant level of absorption that is represented by a generally complex-valued $\epsilon$, and consequently these results do not guarantee that eigenvalues will exist for a given sample to be used as potential target signatures. Thus, we first focus our attention on developing a slight modification of \eqref{stek} in which $\mathcal{S} = \mathcal{S}_\delta$ is a smoothing operator with a positive smoothing parameter $\delta$, which will allow us to use the theory of trace class operators to show that infinitely many eigenvalues of this new problem exist for an absorbing material whenever $\delta$ is sufficiently large. This approach was taken to achieve the same result for acoustic scattering in \cite{cogar2020}. Another result that has so far been unavailable is a careful analysis of the stability of eigenvalues under small perturbations of $\epsilon$, which we will provide for our new problem as well as for \eqref{stek} with $\mathcal{S} = \mathcal{S}_0$. In the context of nondestructive evaluation, this stability property may potentially allow for information about the perturbed medium to be obtained from the observed shift in the eigenvalues relative to a reference set. We remark that the questions of detectability of eigenvalues from measured scattering data and their sensitivity to changes in the medium are not considered here, and we refer to \cite{camano_lackner_monk} for such analysis of \eqref{stek}. \par

The outline of this paper is as follows. In Section \ref{sec_scprob} we recall the necessary Sobolev spaces for our study of Maxwell's equations and present the physical scattering problem that we will consider. We follow in Section \ref{sec_deltastek} with an overview of the Laplace-Beltrami operator on a surface, which we then use to define the smoothing operator $\mathcal{S}_\delta$. This section concludes with the introduction of the auxiliary problem and the resulting eigenvalue problem that we will study for the remainder of our discussion. We begin this investigation in Section \ref{sec_properties} by establishing some basic properties of the eigenvalues, after which we prove our first main result that infinitely many eigenvalues exist when the smoothing parameter $\delta$ is sufficiently large. In Section \ref{sec_pert} we prove the last two main results of this paper. First, we show that the eigenvalues are stable with respect to changes in the medium, including the case from \cite{camano_lackner_monk} when no smoothing is added. Second, we prove that the eigenvalues we consider converge to the eigenvalues of \eqref{stek} as the smoothing parameter $\delta$ converges to zero. Finally, we conclude in Section \ref{sec_conclusion} with some remarks concerning future work in this direction and its applicability to other types of eigenvalue problems arising in scattering theory.

\section{The physical scattering problem} \label{sec_scprob}

Before we introduce the physical scattering problem, we recall the definitions of some basic Sobolev spaces associated with Maxwell's equations. We follow the definitions found in \cite{camano_lackner_monk} for consistency. We let $\mathcal{O}\subset\mathbb{R}^3$ denote a bounded open simply connected domain with Lipschitz boundary $\partial\mathcal{O}$, and we denote the unit outward normal to $\mathcal{O}$ by $\un$. We write the norm for both spaces $L^2(\mathcal{O})$ and $\mathbf{L}^2(\mathcal{O}) := (L^2(\mathcal{O}))^2$ as $\norm{\cdot}_{\mathcal{O}}$, and we write $(\cdot,\cdot)_{\mathcal{O}}$ for their respective inner products. We define the space
\begin{align*}
\Hcurl{\mathcal{O}} &:= \{\mathbf{u}\in\mathbf{L}^2(\mathcal{O}) \mid \curl\mathbf{u}\in\mathbf{L}^2(\mathcal{O})\},
\end{align*}
and we endow this space with the inner product defined by
\begin{equation*}
\inncurl{\mathbf{u}}{\mathbf{u}'}{\mathcal{O}} := (\curl\mathbf{u},\curl\mathbf{u}')_{\mathcal{O}} + (\mathbf{u},\mathbf{u}')_{\mathcal{O}} \quad\forall\mathbf{u},\mathbf{u}'\in\Hcurl{\mathcal{O}}
\end{equation*}
and the corresponding induced norm $\normcurl{\cdot}{\mathcal{O}}$. 
For use in exterior problems, we define the space
\begin{equation*}
\Hcurlloc{\mathbb{R}^3\setminus\overline{\mathcal{O}}} := \{\mathbf{u} \mid \mathbf{u}\in\Hcurl{B_R\setminus\overline{\mathcal{O}}} \quad\forall R>0\}.
\end{equation*}
We will primarily consider the case when the boundary $\partial\mathcal{O}$ is smooth, and we will require the spaces
\begin{align*}
\mathbf{L}_t^2(\partial\mathcal{O}) &:= \{\mathbf{u}\in\mathbf{L}^2(\partial\mathcal{O}) \mid \un\cdot\mathbf{u} = 0 \text{ a.e. on } \partial\mathcal{O}\}, \\
\mathbf{H}_t^s(\partial\mathcal{O}) &:= \{\mathbf{u}\in (H^s(\partial\mathcal{O}))^3 \mid \un\cdot\mathbf{u} = 0 \text{ a.e. on } \partial\mathcal{O}\}, \\
\mathbf{H}^s(\sdiv_{\partial\mathcal{O}},\partial\mathcal{O}) &:= \{\mathbf{u}\in\mathbf{H}_t^s(\partial\mathcal{O}) \mid \sdiv_{\partial\mathcal{O}} \mathbf{u} \in H^s(\partial\mathcal{O})\}, \\
\mathbf{H}^s(\sdiv_{\partial\mathcal{O}}^0,\partial\mathcal{O}) &:= \{\mathbf{u}\in\mathbf{H}^s(\sdiv_{\partial\mathcal{O}},\partial\mathcal{O}) \mid \sdiv_{\partial\mathcal{O}} \mathbf{u} = 0 \text{ on  } \partial\mathcal{O})\}, \\
\mathbf{H}^s(\scurl_{\partial\mathcal{O}}, \partial\mathcal{O}) &:= \{\mathbf{u}\in\mathbf{H}_t^s(\partial\mathcal{O}) \mid \scurl_{\partial\mathcal{O}}\mathbf{u}\in H^s(\partial\mathcal{O})\},
\end{align*}
where $\sdiv_{\partial\mathcal{O}}$ and $\scurl_{\partial\mathcal{O}}$ are the surface divergence and scalar surface curl, respectively, and $s\in\mathbb{R}$. We will also denote the vector surface curl by $\veccurl_{\partial\mathcal{O}}$ and the surface gradient by $\nabla_{\partial\mathcal{O}}$. We note that $\mathbf{H}_t^0(\partial\mathcal{O}) = \mathbf{L}_t^2(\partial\mathcal{O})$, and we define $\mathbf{H}(\sdiv_{\partial\mathcal{O}},\partial\mathcal{O}) := \mathbf{H}^0(\sdiv_{\partial\mathcal{O}},\partial\mathcal{O})$ for convenience. The space $\mathbf{H}_t^s(\partial\mathcal{O})$ is endowed with the standard norm $\norm{\cdot}_{\mathbf{H}_t^s(\partial\mathcal{O})}$, and the spaces $\mathbf{H}^s(\sdiv_{\partial\mathcal{O}},\partial\mathcal{O})$ and $\mathbf{H}^s(\scurl_{\partial\mathcal{O}}, \partial\mathcal{O})$ are endowed with the norms
\begin{align*}
\norm{\mathbf{u}}_{\mathbf{H}^s(\sdiv_{\partial\mathcal{O}},\partial\mathcal{O})}^2 &:= \norm{\mathbf{u}}_{s,\partial\mathcal{O}}^2 + \norm{\sdiv_{\partial\mathcal{O}}\mathbf{u}}_{s,\partial\mathcal{O}}^2, \\
\norm{\mathbf{u}}_{\mathbf{H}^s(\scurl_{\partial\mathcal{O}}, \partial\mathcal{O})}^2 &:= \norm{\mathbf{u}}_{s,\partial\mathcal{O}}^2 + \norm{\scurl_{\partial\mathcal{O}}\mathbf{u}}_{s,\partial\mathcal{O}}^2,
\end{align*}
respectively. We remark that the induced norm on the subspace $\mathbf{H}^s(\sdiv_{\partial\mathcal{O}}^0,\partial\mathcal{O})$ is simply the norm on $\mathbf{H}_t^s(\partial\mathcal{O})$. Finally, we shall briefly require the space of tangential vector fields on the unit sphere defined by
\begin{equation*}
\mathbf{L}_t^2(\mathbb{S}^2) := \{\mathbf{u}:\mathbb{S}^2\to\mathbb{R}^3 \mid \mathbf{u}(\mathbf{d})\cdot\mathbf{d} = 0, \, \mathbf{d}\in\mathbb{S}^2\},
\end{equation*}
where $\mathbb{S}^2 := \{\mathbf{d}\in\mathbb{R}^3 \mid \abs{\mathbf{d}} = 1\}$. For further information and definitions of the surface differential operators introduced above, we refer to \cite{monk}.

We now introduce the physical scattering problem of interest. We consider a function $\epsilon\in L^\infty(\mathbb{R}^3)$ representing the relative electric permittivity of the medium, and we assume that the contrast $1-\epsilon$ is supported in a bounded set $\overline{D}$, where $D$ is a Lipschitz domain with connected complement $\mathbb{R}^3\setminus\overline{D}$. We also assume that $\epsilon|_D$ lies in the space
\begin{equation*}
W_\Sigma^{1,\infty}(D) := \{\mu\in L^\infty(D) \mid \nabla(\mu|_{\Omega_j})\in\mathbf{L}^\infty(\Omega_j), j = 1,2,\dots,J\},
\end{equation*}
where $\{\Omega_j\}_{j=1}^J$ is a partition of $D$ with interface $\Sigma$, and that $\Re(\epsilon)\ge \epsilon_0 > 0$ and $\Im(\epsilon)\ge0$ a.e. in $D$. This regularity condition on $\epsilon$ ensures well-posedness of the subsequent scattering problem (cf. \cite{monk}), and we will also use it to obtain regularity results for Maxwell's equations in our investigation of stability of a certain solution operator with respect to the coefficient $\epsilon$. We consider scattering by this inhomogeneous medium of a time-harmonic incident field $\mathbf{E}^i$ that satisfies the free-space Maxwell's equations
\begin{equation*}
\curl\curl \mathbf{E}^i - k^2 \mathbf{E}^i = \mathbf{0} \text{ in } \mathbb{R}^3
\end{equation*}
for a fixed wave number $k>0$, and we seek a scattered field $\mathbf{E}^s\in\Hcurlloc{\mathbb{R}^3\setminus\overline{D}}$ and a total field $\mathbf{E}\in\Hcurl{D}$ which satisfy
\begin{subequations} \label{sc}
\begin{align}
\curl\curl \mathbf{E}^s - k^2 \mathbf{E}^s &= \mathbf{0} \text{ in } \mathbb{R}^3\setminus\overline{D}, \label{sc1} \\
\curl\curl \mathbf{E} - k^2\epsilon \mathbf{E} &= \mathbf{0} \text{ in } D, \label{sc2} \\
\un\times\mathbf{E} - \un\times\mathbf{E}^s &= \un\times \mathbf{E}^i \text{ on } \partial D, \label{sc3} \\
\un\times\curl \mathbf{E} - \un\times\curl \mathbf{E}^s &= \un\times\curl \mathbf{E}^i \text{ on } \partial D, \label{sc4} \\
\mathclap{\lim_{r\to\infty} \left(\curl \mathbf{E}^s\times\mathbf{x} - ikr\mathbf{E}^s\right) = 0.} \label{sc5}
\end{align}
\end{subequations}
We assume that the Silver-M{\"u}ller radiation condition \eqref{sc5} holds uniformly in all directions, and it follows that \eqref{sc} is well-posed (cf. \cite{colton_kress,monk}). \par

The scattered field $\mathbf{E}^s$ has the asymptotic form of an outgoing spherical wave with a certain amplitude, and for a plane wave incident field
\begin{equation*}
\mathbf{E}^i(\mathbf{x}) = \frac{i}{k} \curl\curl \mathbf{p} e^{-ik\mathbf{x}\cdot\mathbf{d}} = ik(\mathbf{d}\times\mathbf{p})\times\mathbf{d} e^{-ik\mathbf{x}\cdot\mathbf{d}}
\end{equation*}
with direction of propagation $\mathbf{d}\in\mathbb{S}^{2}$ and polarization vector $\mathbf{p}\in\mathbb{R}^3\setminus\{0\}$ such that $\mathbf{p}\perp\mathbf{d}$, we write this asymptotic formula as
\begin{equation} \label{eq:effp}
\mathbf{E}^s(\mathbf{x}) = \frac{e^{ik\abs{\mathbf{x}}}}{\abs{\mathbf{x}}} \left( \mathbf{E}_\infty(\hat{\mathbf{x}},\mathbf{d};\mathbf{p}) + O\left( \frac{1}{\abs{\mathbf{x}}^2} \right) \right) \text{ as } \abs{\mathbf{x}}\to\infty.
\end{equation}
The function $\mathbf{E}_\infty(\hat{\mathbf{x}},\mathbf{d};\mathbf{p})$ is called the \emph{electric far field pattern}, and we refer to $\hat{\mathbf{x}}$, $\mathbf{d}$, and $\mathbf{p}$ as the observation direction, incident direction, and polarization, respectively. When considering the inverse scattering problem, the measurements of this function at various observation and incident directions and polarizations provide the data used to conclude information about the medium under investigation. A central tool in this analysis is the \emph{electric far field operator} $\mathbf{F}:\mathbf{L}_t^2(\mathbb{S}^2) \to \mathbf{L}_t^2(\mathbb{S}^2)$ defined by
\begin{equation*}
(\mathbf{F}\mathbf{g})(\hat{\mathbf{x}}) := \int_{\mathbb{S}^2} \mathbf{E}_\infty(\hat{\mathbf{x}},\mathbf{d}; \mathbf{g}(\mathbf{d}))\, ds(\mathbf{d}), \; \hat{\mathbf{x}}\in\mathbb{S}^2.
\end{equation*}
Due to its dependence upon the electric far field pattern, the electric far field operator $\mathbf{F}$ may be considered as the collected data. As we mentioned in the introduction, our intended application is to use eigenvalues to detect changes in the electric permittivity $\epsilon$ of the medium, which is accomplished by comparing the measured scattering data represented by $\mathbf{F}$ with the computed scattering data for an auxiliary problem that we introduce in the next section.

\section{The auxiliary problem}

\label{sec_deltastek}

We let $B$ be a smooth domain in $\mathbb{R}^3$ with connected boundary $\partial B$ and connected complement $\mathbb{R}^3\setminus\overline{B}$, and we note that $\partial B$ is a smooth closed surface of dimension $2$ without boundary. Before we introduce the auxiliary problem that we will consider, we briefly recall the Laplace-Beltrami operator on $\partial B$, denoted by $\Delta_{\partial B}$, and its relationship to some of the Sobolev spaces that we introduced in Section \ref{sec_scprob}. The scalar Laplace-Beltrami operator is defined as
\begin{equation*}
\Delta_{\partial B} = -\sdiv_{\partial B} \nabla_{\partial B} = \scurl_{\partial B} \veccurl_{\partial B},
\end{equation*}
where we have introduced a negative sign (as in \cite{jost}) in order to ensure nonnegativity of the operator. We summarize the spectral properties of $\Delta_{\partial B}$ in the following theorem (cf. \cite{sayas_brown_hassell}).

\begin{theorem} \label{theorem_LB}

There exists an orthonormal basis $\{Y_m\}_{m=0}^\infty$ of $L^2(\partial B)$ and a nondecreasing divergent sequence of nonnegative real numbers $\{\mu_m\}_{m=0}^\infty$ such that
\begin{equation*}
\Delta_{\partial B} Y_m = \mu_m Y_m, \; m\ge0.
\end{equation*}
The first eigenvalue is $\mu_0 = 0$ with $Y_0 = \abs{\partial B}^{-1/2}$, and $\mu_m>0$ for $m\ge1$.

\end{theorem}

This eigenbasis was used in \cite{cogar2020} to modify the scalar Stekloff eigenvalue problem, and for the present discussion we require the \emph{vector Laplace-Beltrami operator} $\boldsymbol{\Delta}_{\partial B}$ defined as
\begin{equation*}
\boldsymbol{\Delta}_{\partial B} := -\nabla_{\partial B} \sdiv_{\partial B} + \veccurl_{\partial B} \scurl_{\partial B}.
\end{equation*}
We note that we have again included the negative sign in this definition in order to ensure that $\boldsymbol{\Delta}_{\partial B}$ is nonnegative-definite. From the definition of the eigenfunctions $\{Y_m\}$ of the scalar Laplace-Beltrami operator we see that
\begin{align*}
\boldsymbol{\Delta}_{\partial B} \veccurl_{\partial B} Y_m &= \mu_m \veccurl_{\partial B} Y_m, \\
\boldsymbol{\Delta}_{\partial B} \nabla_{\partial B} Y_m &= \mu_m \nabla_{\partial B} Y_m,
\end{align*}
and it follows that for $m\ge1$ the surface vector fields $\veccurl_{\partial B} Y_m$ and $\nabla_{\partial B} Y_m$ are eigenfunctions of $\boldsymbol{\Delta}_{\partial B}$ corresponding to the eigenvalue $\mu_m$. Since the surface $\partial B$ is simply connected, these eigenfunctions (appropriately normalized) constitute an orthonormal basis in $\mathbf{L}_t^2(\partial B)$ (cf. \cite{nedelec}). As a consequence, any tangential vector field $\boldsymbol{\xi}$ defined on $\partial B$ may be expanded in this basis in the form
\begin{equation} \label{eq:vector_expand}
\boldsymbol{\xi} = \sum_{m=1}^\infty \left[ \boldsymbol{\xi}_m^{(1)} \nabla_{\partial B} Y_m + \boldsymbol{\xi}_m^{(2)} \veccurl_{\partial B} Y_m \right].
\end{equation}
It follows that for any $s\in\mathbb{R}$ the space $\mathbf{H}_t^s(\partial B)$ has the spectral characterization
\begin{equation*}
\mathbf{H}_t^s(\partial B) = \left\{\boldsymbol{\xi} \;\middle|\; \sum_{m=1}^\infty \mu_m^{s+1} \left( \abs{\boldsymbol{\xi}_m^{(1)}}^2 + \abs{\boldsymbol{\xi}_m^{(2)}}^2 \right) < \infty \right\},
\end{equation*}
and we may replace the standard norm on this space with the equivalent norm given by
\begin{equation*}
\norm{\boldsymbol{\xi}}_{\mathbf{H}_t^s(\partial B)} := \left[ \sum_{m=1}^\infty \mu_m^{s+1} \left( \abs{\boldsymbol{\xi}_m^{(1)}}^2 + \abs{\boldsymbol{\xi}_m^{(2)}}^2 \right) \right]^{1/2}.
\end{equation*}
We work with the same characterization and equivalent norm of $\mathbf{H}^s(\sdiv_{\partial B}^0,\partial B)$. The spaces $\mathbf{H}^{-1/2}(\sdiv_{\partial B},\partial B)$ and $\mathbf{H}^{-1/2}(\scurl_{\partial B},\partial B)$ are the images of the tangential trace operators $\mathbf{u}\mapsto\un\times\mathbf{u}$ and $\mathbf{u}\mapsto\mathbf{u}_T := (\un\times\mathbf{u})\times\un$ on $\Hcurl{B}$ (cf. \cite{monk}), respectively, and they may be characterized as
\begin{align*}
\mathbf{H}^{-1/2}(\sdiv_{\partial B},\partial B) &= \left\{ \boldsymbol{\xi} \;\middle|\; \sum_{m=1}^\infty \mu_m^{1/2} \left( \mu_m\abs{\boldsymbol{\xi}_m^{(1)}}^2 + \abs{\boldsymbol{\xi}_m^{(2)}}^2 \right) < \infty \right\}, \\
\mathbf{H}^{-1/2}(\scurl_{\partial B},\partial B) &= \left\{ \boldsymbol{\xi} \;\middle|\; \sum_{m=1}^\infty \mu_m^{1/2} \left( \abs{\boldsymbol{\xi}_m^{(1)}}^2 + \mu_m \abs{\boldsymbol{\xi}_m^{(2)}}^2 \right) < \infty \right\}.
\end{align*}
We again replace the standard norms given in Section \ref{sec_scprob} with the equivalent norms
\begin{align*}
\norm{\boldsymbol{\xi}}_{\mathbf{H}^{-1/2}(\sdiv_{\partial B},\partial B)} &:= \left[ \sum_{m=1}^\infty \mu_m^{1/2} \left( \mu_m\abs{\boldsymbol{\xi}_m^{(1)}}^2 + \abs{\boldsymbol{\xi}_m^{(2)}}^2 \right) \right]^{1/2}, \\
\norm{\boldsymbol{\xi}}_{\mathbf{H}^{-1/2}(\scurl_{\partial B},\partial B)} &:= \left[ \sum_{m=1}^\infty \mu_m^{1/2} \left( \abs{\boldsymbol{\xi}_m^{(1)}}^2 + \mu_m \abs{\boldsymbol{\xi}_m^{(2)}}^2 \right) \right]^{1/2}.
\end{align*}

We now proceed to define the operator $\mathcal{S}_0$ as in \cite{camano_lackner_monk}, after which we will provide an equivalent form of the operator in terms of the eigenbasis of $\boldsymbol{\Delta}_{\partial B}$. We define $\mathcal{S}_0: \mathbf{H}^{-1/2}(\scurl_{\partial B},\partial B) \to \mathbf{H}^{1/2}(\sdiv_{\partial B}^0,\partial B)$ by $\mathcal{S}_0 \boldsymbol{\xi} := \veccurl_{\partial B} q$, where $q\in H^1(\partial B)/\mathbb{C}$ is the unique solution of
\begin{equation*}
\Delta_{\partial B} q = \scurl_{\partial B} \boldsymbol{\xi}.
\end{equation*}

In the following proposition we summarize the basic properties of this operator that were established in \cite{camano_lackner_monk}.

\begin{proposition} \label{prop:S0}

The operator $\mathcal{S}_0: \mathbf{H}^{-1/2}(\scurl_{\partial B},\partial B) \to \mathbf{H}^{1/2}(\sdiv_{\partial B}^0,\partial B)$ is bounded and satisfies
\begin{equation*}
\int_{\partial B} \mathcal{S}_0 \mathbf{u}_T\cdot\overline{\mathbf{z}_T} \,ds = \int_{\partial B} \mathcal{S}_0 \mathbf{u}_T\cdot\overline{\mathcal{S}_0\mathbf{z}_T} \,ds = \int_{\partial B} \mathbf{u}_T\cdot\overline{\mathcal{S}_0\mathbf{z}_T} \,ds
\end{equation*}
for all $\mathbf{u},\mathbf{z}\in\mathbf{H}(\textnormal{\textbf{curl}},B)$, where the integrals over $\partial B$ represent duality pairs between $\mathbf{H}^{-1/2}(\sdiv_{\partial B},\partial B)$ and $\mathbf{H}^{-1/2}(\scurl_{\partial B},\partial B)$.

\end{proposition}

Through an eigensystem expansion of $q$ in the definition of $\mathcal{S}_0$, we see that the operator $\mathcal{S}_0$ may be equivalently expressed as
\begin{equation*}
\mathcal{S}_0\boldsymbol{\xi} := \sum_{m=1}^\infty \boldsymbol{\xi}_m^{(2)} \veccurl_{\partial B} Y_m,
\end{equation*}
where $\boldsymbol{\xi}$ has the expansion \eqref{eq:vector_expand}. This form of $\mathcal{S}_0$ motivates the subsequent definition of the smoothing operator $\mathcal{S}_\delta$, which is analogous to the operator introduced in \cite{cogar2020} in the context of the Helmholtz equation. \par

For a given $\delta\ge0$ we define the operator $\mathcal{S}_\delta:\mathbf{H}^{-1/2}(\scurl_{\partial B}, \partial B)\to\mathbf{H}^{1/2}(\div_{\partial B}^0,\partial B)$ as
\begin{equation*}
\mathcal{S}_\delta\boldsymbol{\xi} := \sum_{m=1}^\infty \mu_m^{-\delta} \boldsymbol{\xi}_m^{(2)} \veccurl_{\partial B} Y_m,
\end{equation*}
and we note that for $\delta = 0$ this operator coincides with the operator $\mathcal{S}_0$ defined above. We first observe that $\mathcal{S}_{\delta_1}\mathcal{S}_{\delta_2} = \mathcal{S}_{\delta_1+\delta_2}$ for all $\delta_1,\delta_2\ge0$. We summarize some basic facts of the operator $\mathcal{S}_\delta$ in the following proposition, which are an immediate consequence of the definition of $\mathcal{S}_\delta$ and the spectral characterizations of the related Sobolev spaces provided above.

\begin{proposition} \label{prop:opSdelta}

For any $\rho\ge-\frac{1}{2}$, the operator $\mathcal{S}_\delta$ is bounded from $\mathbf{H}^\rho(\scurl_{\partial B}, \partial B)$ into $\mathbf{H}^{1+\rho+2\delta}(\sdiv_{\partial B}^0,\partial B)$. In particular, the operator $\mathcal{S}_\delta:\mathbf{H}^{-1/2}(\scurl_{\partial B}, \partial B)\to\mathbf{H}^{1/2}(\sdiv_{\partial B}^0,\partial B)$ is compact whenever $\delta>0$. Furthermore, the operator $\mathcal{S}_\delta$ satisfies
\begin{equation} \label{eq:Sdelta_positive}
\int_{\partial B} \mathcal{S}_\delta \mathbf{u}_T\cdot\overline{\mathbf{z}_T} \,ds = \int_{\partial B} \mathcal{S}_{\delta/2} \mathbf{u}_T\cdot\overline{\mathcal{S}_{\delta/2} \mathbf{z}_T} \,ds = \int_{\partial B} \mathbf{u}_T\cdot\overline{\mathcal{S}_\delta\mathbf{z}_T} \,ds
\end{equation}
for all $\mathbf{u},\mathbf{z}\in\mathbf{H}(\textnormal{\textbf{curl}},B)$.

\end{proposition}

For later use we provide the following result concerning the summability of the sequence $\{\mu_m^{-\beta}\}$ for a given $\beta>0$, which follows as a straightforward consequence of Weyl's law (cf. \cite{jost}). We note that this result is valid only for dimension $d=3$.

\begin{proposition} \label{prop:Weyl}

The sequence $\{\mu_m^{-\beta}\}$ is summable if and only if $\beta > 1$.

\end{proposition}

With some basic results in hand, we now define the auxiliary problem that we will use to generate a modification of the electromagnetic Stekloff eigenvalue problem. We assume that $B$ is chosen such that $D\subseteq B$, and we introduce the auxiliary problem of finding
$\mathbf{E}_\lambda^s\in \Hcurlloc{\mathbb{R}^3\setminus\overline{B}}$ satisfying
\begin{subequations} \label{deltaprob}
\begin{align}
\curl\curl \mathbf{E}_\lambda^s - k^2 \mathbf{E}_\lambda^s &= \mathbf{0} \text{ in } \mathbb{R}^3\setminus\overline{B}, \label{deltaprob1} \\
\un\times\curl \mathbf{E}_\lambda^s - \lambda \mathcal{S}_\delta \mathbf{E}_{\lambda,T}^s &= - \un\times\curl \mathbf{E}^i + \lambda \mathcal{S}_\delta \mathbf{E}_T^i \text{ on } \partial B, \label{deltaprob2} \\
\mathclap{\lim_{r\to\infty} \left(\curl \mathbf{E}_\lambda^s\times\mathbf{x} - ikr\mathbf{E}_\lambda^s\right) = 0,} \label{deltaprob3}
\end{align}
\end{subequations}
where the parameter $\lambda\in\mathbb{C}$ satisfies $\Im(\lambda)\ge0$ and will serve as our eigenparameter. If we choose $\delta = 0$, then \eqref{deltaprob} reduces to the standard problem \eqref{stek} with the projection operator $\mathcal{S} = \mathcal{S}_0$. In the case $\delta>0$, the fact that $\mathcal{S}_\delta$ is a bounded operator satisfying \eqref{eq:Sdelta_positive} implies that \eqref{deltaprob} is well-posed whenever $\Im(\lambda)\ge0$ (cf. \cite{camano_lackner_monk}). \par

As for the physical scattering problem, the auxiliary scattered field has an asymptotic expansion of the form \eqref{eq:effp}, and we denote the auxiliary far field pattern by $\mathbf{E}_{\lambda,\infty}^{(\delta)}(\hat{\mathbf{x}},\mathbf{d};\mathbf{p})$. In a similar manner, we define the \emph{auxiliary far field operator} $\mathbf{F}_\lambda^{(\delta)}:\mathbf{L}_t^2(\mathbb{S}^2) \to \mathbf{L}_t^2(\mathbb{S}^2)$ as
\begin{equation*}
(\mathbf{F}_\lambda^{(\delta)}\mathbf{g})(\hat{\mathbf{x}}) := \int_{\mathbb{S}^2} \mathbf{E}_{\lambda,\infty}^{(\delta)}(\hat{\mathbf{x}},\mathbf{d}; \mathbf{g}(\mathbf{d}))\, ds(\mathbf{d}), \; \hat{\mathbf{x}}\in\mathbb{S}^2.
\end{equation*}
We remark that we have explicitly denoted the dependence of the auxiliary far field operator on $\delta$. We now define the \emph{modified far field operator} $\boldsymbol{\mathcal{F}}_\lambda^{(\delta)} := \mathbf{F} - \mathbf{F}_\lambda^{(\delta)}$, which may be written explicitly as
\begin{equation*}
(\boldsymbol{\mathcal{F}}_\lambda^{(\delta)}\mathbf{g})(\hat{\mathbf{x}}) := \int_{\mathbb{S}^2} \left[ \mathbf{E}_\infty(\hat{\mathbf{x}},\mathbf{d}; \mathbf{g}(\mathbf{d})) - \mathbf{E}_{\lambda,\infty}^{(\delta)}(\hat{\mathbf{x}},\mathbf{d}; \mathbf{g}(\mathbf{d})) \right]\, ds(\mathbf{d}), \; \hat{\mathbf{x}}\in\mathbb{S}^2.
\end{equation*}

The modified far field operator $\boldsymbol{\mathcal{F}}_\lambda^{(\delta)}$ serves to compare the measured scattering data to the computed auxiliary data for a given value of the parameter $\lambda$, represented by $\mathbf{F}$ and $\mathbf{F}_\lambda^{(\delta)}$, respectively. By following the same reasoning as in the case $\delta=0$ (cf. \cite{camano_lackner_monk}) we have the following result.

\begin{theorem} \label{theorem:modF}

The modified far field operator $\boldsymbol{\mathcal{F}}_\lambda^{(\delta)}$ is injective provided there exists no nontrivial solution $\mathbf{w}\in\Hcurl{B}$ of the \emph{electromagnetic $\delta$-Stekloff problem}
\begin{subequations} \label{deltaeig}
\begin{align}
\curl\curl \mathbf{w} - k^2 \epsilon \mathbf{w} &= \mathbf{0} \text{ in } B, \label{deltaeig1} \\
\un\times\curl \mathbf{w} - \lambda \mathcal{S}_\delta \mathbf{w}_T &= \mathbf{0} \text{ on } \partial B. \label{deltaeig2}
\end{align}
\end{subequations}

\end{theorem}
We call a value of $\lambda$ for which \eqref{deltaeig} admits a nontrivial solution an \emph{elecromagnetic $\delta$-Stekloff eigenvalue}. We conclude this section with the following assumption on the wave number $k$, which ensures injectivity of a certain solution operator that we will introduce in Section \ref{sec_properties}.

\begin{assumption} \label{assumption:em_injective}

We assume that $k$ is chosen such that there exist no nontrivial solutions $\boldsymbol{\psi}\in\Hcurl{B}$ of the boundary value problem
\begin{subequations} \label{em_odd}
\begin{align}
\nabla\times\nabla\times\boldsymbol{\psi} - k^2\epsilon\boldsymbol{\psi} &= \mathbf{0} \text{ in } B, \label{em_odd1} \\
\sdiv_{\partial B}\left(\nu\times\nabla\times\boldsymbol{\psi}\right) &= 0 \text{ on } \partial B, \label{em_odd2} \\
\scurl_{\partial B} \boldsymbol{\psi}_T &= 0 \text{ on } \partial B. \label{em_odd3}
\end{align}
\end{subequations}

\end{assumption}

We first note that this assumption is automatically satisfied if $\Im(\epsilon)>0$ on an open subset of $D$. Furthermore, we justify this assumption by stating that it holds for all $k>0$ except in a discrete subset, which may be shown by investigating a weak formulation of \eqref{em_odd} and applying the analytic Fredholm theorem (cf. \cite[Theorem 8.26]{colton_kress}) in $\mathbb{C}\setminus\{0\}$. While we keep in mind the intended application to inverse scattering that motivated our consideration of \eqref{deltaeig}, the remainder of our discussion will concern only its spectral properties.

\section{Properties of the electromagnetic $\delta$-Stekloff eigenvalue problem}

\label{sec_properties}

In this section we investigate the properties of the electromagnetic $\delta$-Stekloff eigenvalues, and in order to do so we primarily study a nonhomogeneous version of \eqref{deltaeig} in which we seek $\mathbf{w}\in\Hcurl{B}$ satisfying
\begin{subequations} \label{deltastek}
\begin{align}
\curl\curl \mathbf{w} - k^2 \epsilon \mathbf{w} &= \mathbf{f} \text{ in } B, \label{deltastek1} \\
\un\times\curl \mathbf{w} - \lambda \mathcal{S}_\delta \mathbf{w}_T &= \mathbf{h} \text{ on } \partial B, \label{deltastek2}
\end{align}
\end{subequations}
for given $\mathbf{f}\in\mathbf{L}^2(B)$ and $\mathbf{h}\in\mathbf{H}(\sdiv_{\partial B}^0,\partial B)$. We see that \eqref{deltastek} is equivalent to the weak formulation of finding $\mathbf{w}\in\Hcurl{B}$ such that
\begin{align}
\begin{split} \label{var_deltastek}
&(\curl\mathbf{w},\curl\boldphi)_B - k^2(\epsilon\mathbf{w},\boldphi)_B + \lambda\inner{\mathcal{S}_\delta \mathbf{w}_T}{\boldphi_T}_{\partial B} \\
&\hspace{10em} = (\mathbf{f},\boldphi)_B - \inner{\mathbf{h}}{\boldphi_T}_{\partial B} \quad\forall\boldphi\in\Hcurl{B}.
\end{split}
\end{align}
We note that we have used $\inner{\cdot}{\cdot}_{\partial B}$ to denote the duality pairing of $\mathbf{H}^{-1/2}(\sdiv_{\partial B},\partial B)$ and $\mathbf{H}^{-1/2}(\scurl_{\partial B},\partial B)$ with conjugation in the second argument. We will use the same notation to denote the inner product on $\mathbf{L}_t^2(\partial B)$, and context should prevent any confusion. Since we are only assuming that $\mathbf{f}\in\mathbf{L}^2(B)$, we cannot immediately apply regularity results for Maxwell's equations in order to investigate solvability of \eqref{var_deltastek}, as we have no extra regularity of $\div(\epsilon\mathbf{w})$. Thus, we use the Helmholtz decomposition
\begin{equation*}
\Hcurl{B} = \boldH_0(B)\oplus\nabla H_*^1(B),
\end{equation*}
where we define the space
\begin{equation*}
\boldH_0(B) := \{\mathbf{u}\in\Hcurl{B} \mid \div(\epsilon\mathbf{w}) = 0 \text{ in } B, \; \un\cdot(\epsilon\mathbf{w}) = 0 \text{ on } \partial B\}.
\end{equation*}
We equip this space with the same inner product and norm as $\Hcurl{B}$, and from \cite[Theorem 4.24]{kirsch_hettlich} we observe that the space $\boldH_0(B)$ is compactly embedded into $\mathbf{L}^2(B)$. By writing a solution of \eqref{var_deltastek} as $\mathbf{w} = \mathbf{w}_0 + \nabla\psi$ for $\mathbf{w}_0\in\boldH_0(B)$ and $\psi\in H_*^1(B)$ and restricting the test functions to $\nabla H_*^1(B)$, we see that $\psi\in H_*^1(B)$ must satisfy
\begin{equation} \label{var_psi}
-k^2(\epsilon\nabla\psi,\nabla\psi')_B = (\mathbf{f},\nabla\psi')_B \quad\forall\psi'\in H_*^1(B).
\end{equation}
Well-posedness of \eqref{var_psi} implies that $\psi$ is uniquely determined by $\mathbf{f}$ and that the estimate $\norm{\psi}_{H^1(B)} \le C\norm{\mathbf{f}}_B$ holds. We now return to \eqref{var_deltastek} with the test functions instead restricted to $\boldH_0(B)$, and we observe that $\mathbf{w}_0\in\boldH_0(B)$ must satisfy
\begin{align}
\begin{split} \label{var0_deltastek}
&(\curl\mathbf{w}_0,\curl\boldphi_0)_B - k^2(\epsilon\mathbf{w}_0,\boldphi_0)_B + \lambda\inner{\mathcal{S}_\delta \mathbf{w}_{0,T}}{\boldphi_{0,T}}_{\partial B} \\
&\hspace{8em} = (\mathbf{f}+k^2\epsilon\nabla\psi,\boldphi_{0,T})_B - \inner{\mathbf{h}}{\boldphi_{0,T}}_{\partial B} \quad\forall\boldphi_0\in\boldH_0(B).
\end{split}
\end{align}
Conversely, solutions $\psi$ and $\mathbf{w}_0$ of \eqref{var_psi} and \eqref{var0_deltastek}, respectively, yield a solution $\mathbf{w} = \mathbf{w} + \nabla\psi$ of \eqref{deltastek}, and we use this equivalence in the following theorem to investigate when this problem is well-posed.

\begin{theorem} \label{theorem:Fredholm}

The nonhomogeneous $\delta$-Stekloff problem \eqref{deltastek} is of Fredholm type. In particular, if $\lambda$ is not an electromagnetic $\delta$-Stekloff eigenvalue, then there exists a unique solution $\mathbf{w}\in\Hcurl{B}$ of \eqref{deltastek} satisfying the estimate
\begin{equation*}
\normcurl{\mathbf{w}}{B} \le C\left( \norm{\mathbf{f}}_B + \norm{\mathbf{h}}_{\mathbf{H}(\sdiv_{\partial B}^0,\partial B)} \right).
\end{equation*}
Furthermore, the electromagnetic $\delta$-Stekloff eigenvalues form a discrete subset of $\mathbb{C}$ without finite accumulation point.

\end{theorem}

\begin{proof}

By means of the Riesz representation theorem we define the operators $\mathbb{\hat{A}}, \mathbb{B}_\lambda:\boldH_0(B)\to\boldH_0(B)$ such that
\begin{align*}
(\mathbb{\hat{A}}\mathbf{u},\boldphi_0)_{\boldH_0(B)} &= (\curl\mathbf{u},\curl\boldphi_0)_B + k^2(\mathbf{u},\boldphi_0)_B, \\
(\mathbb{B}_\lambda \mathbf{u},\boldphi_0)_{\boldH_0(B)} &= -k^2((1+\epsilon) \mathbf{u},\boldphi_0)_B + \lambda\inner{\mathcal{S}_\delta \mathbf{u}_T}{\boldphi_{0,T}}_{\partial B},
\end{align*}
for all $\mathbf{u},\boldphi_0\in\boldH_0(B)$, and we observe that \eqref{var0_deltastek} is equivalent to finding $\mathbf{w}_0\in\boldH_0(B)$ for which
\begin{equation*}
((\mathbb{\hat{A}} + \mathbb{B}_\lambda)\mathbf{w}_0,\boldphi_0)_{\boldH_0(B)} = (\mathbf{f}+k^2\epsilon\nabla\psi,\boldphi_{0,T})_B - \inner{\mathbf{h}}{\boldphi_{0,T}}_{\partial B} \quad\forall\boldphi_0\in\boldH_0(B).
\end{equation*}
As a consequence, it suffices to study the operator $\mathbb{\hat{A}} + \mathbb{B}_\lambda$. First, we see that $\mathbb{\hat{A}}$ is defined in terms of an equivalent inner product on $\boldH_0(B)$, and it follows from the Riesz representation theorem that the operator $\mathbb{\hat{A}}$ must be invertible. Second, for each $\mathbf{u}\in\boldH_0(B)$ we see that
\begin{align}
\norm{\mathbb{B}_\lambda \mathbf{u}}_{\boldH_0(B)} &= \sup_{\substack{\boldphi_0\in\boldH_0(B) \\ \norm{\boldphi_0}_{\boldH_0(B)}\le1}} \abs{-k^2((1+\epsilon) \mathbf{u},\boldphi_0)_B + \lambda\inner{\mathcal{S}_\delta \mathbf{u}_T}{\boldphi_{0,T}}_{\partial B}} \nonumber\\
&\le C\left( \norm{\mathbf{u}}_B + \norm{\mathcal{S}_\delta \mathbf{u}_T}_{\mathbf{H}(\sdiv_{\partial B}^0,\partial B)} \right), \label{eq:B_lambda}
\end{align}
where the constant $C>0$ depends only on $k$, $\epsilon$, $\lambda$, and $B$. If a sequence $\{\mathbf{u}_m\}$ converges weakly in $\boldH_0(B)$ to $\mathbf{u}_0\in\boldH_0(B)$, then the compact embedding of $\boldH_0(B)$ into $\mathbf{L}^2(B)$ implies that $\mathbf{u}_m\to\mathbf{u}_0$ in $\mathbf{L}^2(B)$. Moreover, boundedness of $\mathcal{S}_\delta$ into $\mathbf{H}^{1/2}(\sdiv_{\partial B}^0,\partial B)$ implies that $\mathcal{S}_\delta \mathbf{u}_{m,T}\rightharpoonup \mathcal{S}_\delta \mathbf{u}_{0,T}$ in $\mathbf{H}_t^{1/2}(\partial B)$, and from the compact embedding of the latter space into $\mathbf{L}_t^2(\partial B)$ we obtain $\mathcal{S}_\delta \mathbf{u}_{m,T}\to \mathcal{S}_\delta \mathbf{u}_{0,T}$ in $\mathbf{H}(\sdiv_{\partial B}^0,\partial B)$. Thus, we see from \eqref{eq:B_lambda} that $\mathbb{B}_\lambda \mathbf{u}_m \to \mathbb{B}_\lambda \mathbf{u}_0$ in $\boldH_0(B)$, and it follows that $\mathbb{B}_\lambda$ is compact. By combining these results, we conclude that the operator $\mathbb{\hat{A}} + \mathbb{B}_\lambda$ is a Fredholm operator of index zero, which implies that \eqref{var0_deltastek} is of Fredholm type. The aforementioned equivalence implies that the same holds for \eqref{deltastek}. In particular, this problem is well-posed whenever $\lambda$ is not an electromagnetic $\delta$-Stekloff eigenvalue. \par


We shall now establish discreteness of the eigenvalues. If $(\lambda,\mathbf{w})$ is an eigenpair of \eqref{deltaeig}, then $\mathbf{w}$ satisfies \eqref{var0_deltastek} with $\mathbf{f}=\mathbf{0}$, $\psi=0$, and $\mathbf{h}=\mathbf{0}$. Taking the imaginary part of both sides of this equation with $\boldphi_0 = \mathbf{w}$ (and applying the results of Proposition \ref{prop:opSdelta}) yields
\begin{equation*}
-k^2(\Im(\epsilon)\mathbf{w},\mathbf{w})_B + \Im(\lambda)\norm{\mathcal{S}_{\delta/2} \mathbf{w}_T}_{\mathbf{H}(\sdiv_{\partial B}^0,\partial B)} = 0.
\end{equation*}
If $\mathcal{S}_{\delta/2} \mathbf{w}_T = 0$, then we see that $\mathbf{w}$ satisfies \eqref{em_odd}, and by Assumption \ref{assumption:em_injective} we would have $\mathbf{w} = \mathbf{0}$ in $B$. Since the eigenfunction $\mathbf{w}$ must be nontrivial, we must have $\mathcal{S}_{\delta/2} \mathbf{w}_T \neq 0$, and we may solve for $\Im(\lambda)$ in order to obtain
\begin{equation*}
\Im(\lambda) = \frac{k^2(\Im(\epsilon)\mathbf{w},\mathbf{w})_B}{\norm{\mathcal{S}_{\delta/2} \mathbf{w}_T}_{\mathbf{H}(\sdiv_{\partial B}^0,\partial B)}} \ge 0.
\end{equation*}
Thus, we see that every eigenvalue satisfies $\Im(\lambda)\ge0$, and in particular we have shown that the operator $\mathbb{\hat{A}} + \mathbb{B}_\lambda$ is injective, and hence invertible, whenever $\Im(\lambda)<0$. Since the operator $\mathbb{B}_\lambda$ depends analytically on $\lambda$, the analytic Fredholm theorem now implies that $\mathbb{\hat{A}} + \mathbb{B}_\lambda$ is invertible for all $\lambda$ except in a discrete subset of $\mathbb{C}$ without finite accumulation point (cf. \cite[Theorem 8.26]{colton_kress}), which implies that the set of electromagnetic $\delta$-Stekloff eigenvalues is discrete. 
\end{proof}

In order to establish further properties of the electromagnetic $\delta$-Stekloff eigenvalues, we define the operator $\mathbf{T}_z^{(\delta)}:\mathbf{H}(\sdiv_{\partial B}^0,\partial B)\to \mathbf{H}(\sdiv_{\partial B}^0,\partial B)$ by $\mathbf{T}_z^{(\delta)}\mathbf{h} := \mathcal{S}_0\mathbf{w}_T$, where $\mathbf{w}\in\Hcurl{B}$ satisfies \eqref{deltastek} with $\lambda = z$ and $\mathbf{f}=\mathbf{0}$, i.e.
\begin{subequations} \label{defT}
\begin{align}
\curl\curl \mathbf{w} - k^2 \epsilon \mathbf{w} &= \mathbf{0} \text{ in } B, \label{defT1} \\
\un\times\curl \mathbf{w} - z \mathcal{S}_\delta \mathbf{w}_T &= \mathbf{h} \text{ on } \partial B. \label{defT2}
\end{align}
\end{subequations}
Here $z\in\mathbb{R}$ is chosen to be outside the set of electromagnetic $\delta$-Stekloff eigenvalues, which implies that $\mathbf{T}_z^{(\delta)}$ is well-defined by Theorem \ref{theorem:Fredholm}, and we note that we may choose such as a value of $z$ due to discreteness of the eigenvalues. We follow similar reasoning as in \cite[Lemma 3.4]{camano_lackner_monk} to obtain the following regularity property of $\mathbf{T}_z^{(\delta)}$.

\begin{proposition} \label{prop:ST}

For any $\delta\ge0$ the operator $\mathbf{T}_z^{(\delta)}: \mathbf{H}(\sdiv_{\partial B}^0,\partial B)\to \mathbf{H}^1(\sdiv_{\partial B}^0,\partial B)$ is bounded.

\end{proposition}

\begin{proof}

For a given $\mathbf{h}\in\mathbf{H}(\sdiv_{\partial B}^0,\partial B)$, we let $\mathbf{w}$ satisfy \eqref{defT}, and it follows that $\mathbf{T}_z^{(\delta)} \mathbf{h} = \mathcal{S}_0 \mathbf{w}_T$. We see from the boundary condition \eqref{defT2} that $\un\times\curl\mathbf{w}\in\mathbf{L}_t^2(\partial B)$, and since $\curl\mathbf{w}\in\Hcurl{B}\cap\Hdiv{B}$ we obtain $\un\cdot\curl\mathbf{w}\in L^2(\partial B)$ (cf. \cite{costabel}). From the identity $\scurl_{\partial B} \mathbf{w}_T = \un\cdot\curl\mathbf{w}$ we have $\scurl_{\partial B}\mathbf{w}_T\in L^2(\partial B)$. Recalling the definition of the operator $\mathcal{S}_0$, there exists a unique $q\in H^1(\partial B)/\mathbb{C}$ satisfying $\Delta_{\partial B} q = \scurl_{\partial B} \mathbf{w}_T$ for which $\mathcal{S}_0 \mathbf{w}_T = \veccurl_{\partial B} q$. We observe that $\Delta_{\partial B} q \in L^2(\partial B)$, which implies that $\veccurl_{\partial B} q\in \mathbf{H}^1(\sdiv_{\partial B}^0,\partial B)$ with the estimate
\begin{equation*}
\norm{\veccurl_{\partial B} q}_{\mathbf{H}^1(\sdiv_{\partial B}^0,\partial B)}\le C\norm{\scurl_{\partial B} \mathbf{w}_T}_{L^2(\partial B)}.
\end{equation*}
By the regularity estimate from \cite{costabel} and well-posedness of \eqref{defT} we conclude that $\mathbf{T}_z^{(\delta)}$ is bounded into $\mathbf{H}^1(\sdiv_{\partial B}^0,\partial B)$.
\end{proof}

The operator $\mathbf{T}_0^{(0)}$ was used in \cite{camano_lackner_monk} in order to obtain results on the standard electromagnetic Stekloff eigenvalues. However, for $\delta>0$ the spectrum of this operator no longer has a clear relationship to the set of $\delta$-Stekloff eigenvalues, and we must instead consider the operator $\boldPsi_z^{(\delta)}:\mathbf{H}(\div_{\partial B}^0,\partial B)\to \mathbf{H}(\div_{\partial B}^0,\partial B)$ defined by $\boldPsi_z^{(\delta)} := \mathcal{S}_{\delta/2} \mathbf{T}_z^{(\delta)} \mathcal{S}_{\delta/2}$. We remark that Proposition \ref{prop:ST} and the smoothing property of $\mathcal{S}_\delta$ stated in Proposition \ref{prop:opSdelta} immediately imply that the operator $\boldPsi_z^{(\delta)}$ is bounded into $\mathbf{H}^{1+\delta}(\sdiv_{\partial B}^0,\partial B)$ and hence must be compact. We follow the same reasoning as in \cite[Lemma 4.4]{cogar2020} to establish the following relationship between the set of electromagnetic $\delta$-Stekloff eigenvalues and the spectrum of $\boldPsi_z^{(\delta)}$. For future use, we prove a slightly more general result.

\begin{proposition} \label{prop:Psi}

For a given $\delta\ge0$, let $\delta_1, \delta_2$ be nonnegative numbers such that $\delta_1 + \delta_2 = \delta$. Then a given $\lambda\in\mathbb{C}$ is an electromagnetic $\delta$-Stekloff eigenvalue if and only if $(\lambda - z)^{-1}$ is an eigenvalue of the operator $\mathcal{S}_{\delta_1} \mathbf{T}_z^{(\delta)} \mathcal{S}_{\delta_2}$.

\end{proposition}

\begin{proof}

We first note that $\mathcal{S}_\delta = \mathcal{S}_{\delta_2} \mathcal{S}_{\delta_1}$. We suppose that $\lambda$ is an electromagnetic $\delta$-Stekloff eigenvalue with eigenfunction $\mathbf{w}$, and we rewrite the boundary condition \eqref{deltaeig2} as
\begin{equation*}
\un\times\curl\mathbf{w} - z\mathcal{S}_\delta \mathbf{w}_T = (\lambda-z)\mathcal{S}_\delta \mathbf{w}_T \text{ on } \partial B.
\end{equation*}
If we define $\mathbf{h} := (\lambda-z)\mathcal{S}_{\delta_1} \mathbf{w}_T$, then it follows from the definition of $\mathbf{T}_z^{(\delta)}$ that $\mathcal{S}_0\mathbf{w}_T = \mathbf{T}_z^{(\delta)}\mathcal{S}_{\delta_2}\mathbf{h}$, and we observe that
\begin{equation*}
\mathcal{S}_{\delta_1}\mathbf{T}_z^{(\delta)}\mathcal{S}_{\delta_2}\mathbf{h} = \mathcal{S}_{\delta_1} \mathbf{w}_T = (\lambda-z)^{-1}\mathbf{h}.
\end{equation*}
As in the proof of Theorem \ref{theorem:Fredholm}, we see that Assumption \ref{assumption:em_injective} implies that $\mathbf{h}\neq\mathbf{0}$, and as a result we conclude that $(\lambda-z)^{-1}$ is an eigenvalue of the operator $\mathcal{S}_{\delta_1} \mathbf{T}_z^{(\delta)} \mathcal{S}_{\delta_2}$. Following the same steps in reverse order yields the converse.
\end{proof}

In particular, by choosing $\delta_1 = \delta_2 = \frac{\delta}{2}$ in Proposition \ref{prop:Psi} we see that $\lambda$ is an electromagnetic $\delta$-Stekloff eigenvalue if and only if $(\lambda-z)^{-1}$ is an eigenvalue of $\boldPsi_z^{(\delta)}$. Before proceeding further, we state the following lemma as an immediate consequence of Assumption \ref{assumption:em_injective}.

\begin{lemma} \label{lemma:em_nullspace}

The operator $\boldPsi_z^{(\delta)}$ is injective.

\end{lemma}

%

We now turn our attention to showing existence of eigenvalues, and from Proposition \eqref{prop:Psi} we see that it suffices to establish existence of eigenvalues of the operator $\boldPsi_z^{(\delta)}$.
If $\epsilon$ is real-valued, then by the same reasoning in \cite[Lemma 3.5]{camano_lackner_monk} we see that $\boldPsi_z^{(\delta)}$ is a compact self-adjoint operator on the Hilbert space $\mathbf{H}(\sdiv_{\partial B}^0,\partial B)$, and we obtain the following result from the Hilbert-Schmidt theorem and Lemma \ref{lemma:em_nullspace}.

\begin{theorem} \label{theorem:real_eps}

If $\epsilon$ is real-valued, then there exist infinitely many electromagnetic $\delta$-Stekloff eigenvalues, and all eigenvalues are real.

\end{theorem}

However, in the case that $\epsilon$ is not real-valued, the operator $\boldPsi_z^{(\delta)}$ is not self-adjoint and no existence results are known for the standard electromagnetic Stekloff eigenvalues. We now proceed to show that for large enough $\delta$ the operator $\boldPsi_z^{(\delta)}$ is trace class, as defined below, and that we may apply Lidski's Theorem in order to conclude the existence of infinitely many eigenvalues even for complex-valued $\epsilon$. We first define trace class operators and state Lidski's Theorem (cf. \cite{colton_kress,ringrose}).

\begin{definition} \label{definition:trace_class}

An operator $T$ on a Hilbert space is a \emph{trace class operator} if there exists a sequence of operators $\{T_m\}$ for which $T_m$ has rank no greater than $m$ and
\begin{equation*}
\sum_{m=1}^\infty \norm{T - T_m} < \infty.
\end{equation*}

\end{definition}

\begin{theorem} \label{theorem:lidski} \textnormal{(Lidski's Theorem)} If $T$ is a trace class operator on a Hilbert space $X$ such that $T$ has finite-dimensional nullspace and $\Im(Tg,g)_X\ge0$ for each $g\in X$, then $T$ has an infinite number of eigenvalues.
\end{theorem}


We have already established in Lemma \ref{lemma:em_nullspace} that $\boldPsi_z^{(\delta)}$ is injective, which implies that the operator has finite-dimensional nullspace. In the next two lemmas we verify the remaining hypotheses of Lidski's Theorem, but with a slight change. In particular, we will apply the result to the operator $-\boldPsi_z^{(\delta)}$.

\begin{lemma} \label{lemma:em_imag_part}

The operator $-\boldPsi_z^{(\delta)}:\mathbf{H}(\sdiv_{\partial B}^0,\partial B)\to \mathbf{H}(\sdiv_{\partial B}^0,\partial B)$ satisfies
\begin{equation*}
\Im\inner{-\boldPsi_z^{(\delta)} \mathbf{h}}{\mathbf{h}}_{\partial B} \ge 0
\end{equation*}
for all $\mathbf{h}\in\mathbf{H}(\div_{\partial B}^0,\partial B)$.

\end{lemma}

\begin{proof}

We first observe that, by our definition of $\boldPsi_z^{(\delta)}$ and the results of Proposition \ref{prop:opSdelta}, we have
\begin{equation*}
\Im\inner{-\boldPsi_z^{(\delta)} \mathbf{h}}{\mathbf{h}}_{\partial B} = \Im\inner{-\mathbf{T}_z^{(\delta)} \mathcal{S}_{\delta/2} \mathbf{h}}{\mathcal{S}_{\delta/2} \mathbf{h}}_{\partial B} \quad\forall\mathbf{h}\in\mathbf{H}(\sdiv_{\partial B}^0,\partial B),
\end{equation*}
and hence it suffices to establish the nonnegativity condition for the operator $-\mathbf{T}_z^{(\delta)}$. For a given $\mathbf{h}$ we let $\mathbf{w}$ denote the solution of \eqref{defT}, which implies that $\mathbf{T}_z^{(\delta)}\mathbf{h} = \mathcal{S}_0 \mathbf{w}_T$. From \eqref{var_deltastek} we see that
\begin{align*}
\inner{-\mathbf{T}_z^{(\delta)} \mathbf{h}}{\mathbf{h}}_{\partial B} &= -\inner{\mathcal{S}_0\mathbf{w}_T}{\mathbf{h}}_{\partial B} \\
 &= -\inner{\mathbf{w}_T}{\mathcal{S}_0\mathbf{h}}_{\partial B} \\
&= -\overline{\inner{\mathbf{h}}{\mathbf{w}_T}_{\partial B}} \\
&= (\curl\mathbf{w},\curl\mathbf{w})_B - k^2(\overline{\epsilon}\mathbf{w},\mathbf{w})_B + z\inner{\mathcal{S}_\delta\mathbf{w}_T}{\mathbf{w}_T}_{\partial B},
\end{align*}
where we have used the fact that $\mathcal{S}_0 \mathbf{h} = \mathbf{h}$, $z$ was chosen to be real, and $\inner{\mathcal{S}_\delta\mathbf{w}_T}{\mathbf{w}_T}_{\partial B}$ is real as a consequence of Proposition \ref{prop:opSdelta}. Considering the imaginary part yields
\begin{equation*}
\Im\inner{-\mathbf{T}_z^{(\delta)} \mathbf{h}}{\mathbf{h}}_{\partial B} = -k^2(\Im(\overline{\epsilon})\mathbf{w},\mathbf{w})_B = k^2(\Im(\epsilon)\mathbf{w},\mathbf{w})_B \ge 0
\end{equation*}
due to our assumption that $\Im(\epsilon)\ge0$, and we arrive at the desired result.
\end{proof}

We now proceed to establish that, for $\delta>1$, the operator $\boldsymbol{\Psi}_z^{(\delta)}$ is trace class, as we defined in Definition \ref{definition:trace_class}. We rely on the regularity result that we obtained in Proposition \ref{prop:ST}, but we remark that it may be possible to improve this result with a more careful analysis of the regularity properties of \eqref{defT}.


\begin{lemma} \label{lemma:em_trace_class}

If $\delta > 1$, then $\boldsymbol{\Psi}_z^{(\delta)}$ is a trace class operator.

\end{lemma}

\begin{proof}

For each $M\in\mathbb{N}$ we define the operator $\mathcal{I}^{(M)}:\mathbf{H}(\sdiv_{\partial B}^0,\partial B)\to \mathbf{H}(\sdiv_{\partial B}^0,\partial B)$ by
\begin{equation*}
\mathcal{I}^{(M)}\boldsymbol{\xi} := \sum_{m=1}^{M-1} \boldsymbol{\xi}_m^{(2)}\veccurl_{\partial B} Y_m,
\end{equation*}
and we note that $\mathcal{I}^{(M)}$ has rank $M-1$. It follows that the operator $\mathcal{I}^{(M)}\boldsymbol{\Psi}_z^{(\delta)}$ has rank no greater than $M$. For a given $\mathbf{h}\in\mathbf{H}(\sdiv_{\partial B}^0,\partial B)$ we denote by $\mathbf{w}$ the solution of \eqref{defT} with right-hand side $\mathcal{S}_{\delta/2} \mathbf{h}$, implying that $\boldPsi_z^{(\delta)}\mathbf{h} = \mathcal{S}_{\delta/2} \mathbf{w}_T$, and we have
\begin{align*}
\norm{\left( \boldsymbol{\Psi}_z^{(\delta)} - \mathcal{I}^{(M)}\boldsymbol{\Psi}_z^{(\delta)} \right)\mathbf{h}}_{\mathbf{H}(\div_{\partial B}^0,\partial B)} &= \norm{\left( \boldsymbol{\Psi}_z^{(\delta)} - \mathcal{I}^{(M)}\boldsymbol{\Psi}_z^{(\delta)} \right)\mathbf{h}}_{\mathbf{L}_t^2(\partial B)} \\
&= \left( \sum_{m=M}^\infty \mu_m \abs{\mu_m^{-\delta/2}(\mathbf{w}_T)_m^{(2)}}^2 \right)^{1/2} \\
&= \left( \sum_{m=M}^\infty \mu_m^{ -\left(1+\delta\right) } \mu_m^2 \abs{(\mathbf{w}_T)_m^{(2)}}^2 \right)^{1/2} \\
&\le \mu_M^{-\frac{1}{2}\left(1 + \delta\right)} \norm{\mathcal{S}_0\mathbf{w}_T}_{\mathbf{H}^1(\sdiv_{\partial B}^0,\partial B)} \\
&\le C \mu_M^{-\frac{1}{2}\left(1 + \delta\right)} \norm{\mathbf{h}}_{\mathbf{H}(\sdiv_{\partial B}^0,\partial B)},
\end{align*}
where the final estimate follows from Proposition \ref{prop:ST} and boundedness of $\mathcal{S}_\delta$. As a result we obtain
\begin{equation*}
\norm{\boldsymbol{\Psi}_z^{(\delta)} - \mathcal{I}^{(M)}\boldsymbol{\Psi}_z^{(\delta)}} \le C \mu_M^{-\frac{1}{2}\left(1 + \delta\right)} \quad\forall M\in\mathbb{N}.
\end{equation*}
By Proposition \ref{prop:Weyl} we know that $\{\mu_m^{-\beta}\}$ is summable if and only if $\beta > 1$, and it follows from Definition \ref{definition:trace_class} that $\boldsymbol{\Psi}_z^{(\delta)}$ is trace class whenever $\delta>1$.
\end{proof}

By combining the results of Lemmas \ref{lemma:em_nullspace}, \ref{lemma:em_imag_part}, and \ref{lemma:em_trace_class} and applying Lidski's Theorem to the operator $-\boldPsi_z^{(\delta)}$, we obtain the following result on the electromagnetic $\delta$-Stekloff eigenvalues as a consequence of Proposition \ref{prop:Psi}.

\begin{theorem} \label{theorem:existence}

If $\delta>1$, then there exist infinitely many electromagnetic $\delta$-Stekloff eigenvalues.

\end{theorem}



%

\section{Stability of $\delta$-Stekloff eigenvalues}

\label{sec_pert}

With our application of nondestructive testing of materials in mind, we devote this section to the investigation of stability of the eigenvalues. We first consider stability with respect to changes in $\epsilon$, which ensures that small changes in a material do not cause large deviations of the eigenvalues; while shifts in the eigenvalues are desired in order to detect such changes in the material, it is preferable that the eigenvalues for the perturbed medium remain near those for the unperturbed medium in order to facilitate a careful analysis of the perturbation. Second, we will consider stability with respect to $\delta$. Since we have complete control over this smoothing parameter, our primary interest here is to show that the electromagnetic $\delta$-Stekloff eigenvalues converge to the standard electromagnetic Stekloff eigenvalues as $\delta\to0^+$. In both cases we will conclude stability of the eigenvalues from a convergence result for the corresponding solution operators.

\subsection{Stability with respect to $\epsilon$} \label{subsec_pert_epsilon}

We show that electromagnetic $\delta$-Stekloff eigenvalues are stable with respect to small changes in the relative permittivity $\epsilon$. We first remark that, in order to apply regularity results for Maxwell's equations, we will consider the Sobolev space $\mathbf{H}^s(B)$ for values of $s$ in the interval $(0,\frac{1}{2})$, and we refer to \cite{bonito_guermond_luddens} for the definition of such spaces. In this section we write the solution operator $\boldPsi_z^{(\delta)}$ as $\boldPsi_{z,\epsilon}^{(\delta)}$ in order to make explicit its dependence on the relative permittivity $\epsilon$, and we establish continuity of the mapping $\epsilon\mapsto\boldPsi_{z,\epsilon}^{(\delta)}$ through an operator factorization of $\boldPsi_{z,\epsilon}^{(\delta)}$ that we now derive. We recall that, for a given $\mathbf{h}\in\mathbf{H}(\sdiv_{\partial B}^0,\partial B)$, we have defined $\boldPsi_{z,\epsilon}^{(\delta)} \mathbf{h} = \mathcal{S}_{\delta/2} \mathbf{w}_T$, where $\mathbf{w}\in\Hcurl{B}$ is the unique solution of
\begin{subequations} \label{W}
\begin{align}
\curl\curl\mathbf{w} - k^2 \epsilon \mathbf{w} &= \mathbf{0} \text{ in } B, \label{W1} \\
\un\times\curl\mathbf{w} - z\mathcal{S}_\delta \mathbf{w}_T &= \mathcal{S}_{\delta/2} \mathbf{h} \text{ on } \partial B. \label{W2}
\end{align}
\end{subequations}
If $\mathbf{w}_j$ satisfies \eqref{W} for $\epsilon = \epsilon_j$, $j=0,1$, then we see that $\mathbf{v} := \mathbf{w}_1 - \mathbf{w}_0 \in \Hcurl{B}$ satisfies
\begin{subequations} \label{V}
\begin{align}
\curl\curl\mathbf{v} - k^2 \epsilon \mathbf{v} &= \mathbf{f} \text{ in } B, \label{V1} \\
\un\times\curl\mathbf{v} - z\mathcal{S}_\delta \mathbf{v}_T &= \mathbf{0} \text{ on } \partial B, \label{V2}
\end{align}
\end{subequations}
with $\epsilon = \epsilon_1$ and $\mathbf{f} = k^2(\epsilon_1-\epsilon_0)\mathbf{w}_0$. \par

This observation motivates us to define the following operators. First, we define $\mathbf{W}_\epsilon:\mathbf{H}(\sdiv_{\partial B}^0,\partial B)\to\Hcurl{B}$ such that $\mathbf{W}_\epsilon\mathbf{h} := \mathbf{w}$, where $\mathbf{w}$ satisfies \eqref{W}. Given a reference permittivity $\epsilon_0$ that we view as fixed and a perturbed permittivity $\epsilon_1$, we define the multiplication operator $\mathbf{M}_{\epsilon_1,\epsilon_0}:\mathbf{L}^2(B)\to\mathbf{L}^2(B)$ by $\mathbf{M}_{\epsilon_1,\epsilon_0} \mathbf{f} := (\epsilon_1-\epsilon_0)\mathbf{f}$. Finally, we define the operator $\mathbf{V}_\epsilon:\mathbf{L}^2(B)\to\Hcurl{B}$ such that $\mathbf{V}_\epsilon\mathbf{f} := \mathbf{v}$, where $\mathbf{v}$ satisfies \eqref{V}. We remark that the solutions operators $\mathbf{W}_\epsilon$ and $\mathbf{V}_\epsilon$ are well-defined as a result of Theorem \ref{theorem:Fredholm} and the assumption that $z$ is not an electromagnetic $\delta$-Stekloff eigenvalue. Since $\mathbf{w}_j = \mathbf{W}_{\epsilon_j} \mathbf{h}$, $j = 0,1$, it follows that
\begin{equation*}
(\mathbf{W}_{\epsilon_1} - \mathbf{W}_{\epsilon_0}) \mathbf{h} = \mathbf{V}_{\epsilon_1} \biggr[ k^2(\epsilon_1-\epsilon_0)\mathbf{w}_0 \biggr] = k^2\mathbf{V}_{\epsilon_1} \mathbf{M}_{\epsilon_1,\epsilon_0} \mathbf{W}_{\epsilon_0} \mathbf{h}.
\end{equation*}
Since this equation holds for all choices of $\mathbf{h}$, we arrive at the factorization
\begin{equation*}
\mathbf{W}_{\epsilon_1} - \mathbf{W}_{\epsilon_0} = k^2\mathbf{V}_{\epsilon_1} \mathbf{M}_{\epsilon_1,\epsilon_0} \mathbf{W}_{\epsilon_0}.
\end{equation*}
From the relationship $\boldPsi_{z,\epsilon}^{(\delta)} = \mathcal{S}_{\delta/2} \boldsymbol{\gamma}_T \mathbf{W}_\epsilon$, where $\boldsymbol{\gamma}_T:\Hcurl{B}\to\mathbf{H}^{-1/2}(\scurl_{\partial B},\partial B)$ is the tangential trace operator $\boldsymbol{\gamma}_T\mathbf{u} := \mathbf{u}_T$, we immediately obtain the estimate
\begin{equation} \label{psi_est1}
\norm{\boldPsi_{z,\epsilon_1}^{(\delta)} - \boldPsi_{z,\epsilon_0}^{(\delta)}} \le C\norm{\mathbf{V}_{\epsilon_1}} \norm{\mathbf{M}_{\epsilon_1,\epsilon_0} \mathbf{W}_{\epsilon_0}},
\end{equation}
in which the constant $C>0$ is independent of both $\epsilon_1$ and $\epsilon_0$. We provide a further estimate of the right-hand side of \eqref{psi_est1} in the following lemma.

\begin{lemma} \label{lemma:MW_est}

The operator $\mathbf{M}_{\epsilon_1,\epsilon_0} \mathbf{W}_{\epsilon_0}:\mathbf{H}(\sdiv_{\partial B}^0,\partial B)\to\mathbf{L}^2(B)$ satisfies the norm estimate
\begin{equation} \label{MW_est}
\norm{\mathbf{M}_{\epsilon_1,\epsilon_0} \mathbf{W}_{\epsilon_0}} \le C_{s,\epsilon_0} \norm{\epsilon_1-\epsilon_0}_{L^{3/s}(B)},
\end{equation}
where the constants $s\in(0,\frac{1}{2})$ and $C_{s,\epsilon_0}>0$ are independent of $\epsilon_1$.

\end{lemma}

\begin{proof}

For a given $\mathbf{h}\in\mathbf{H}(\sdiv_{\partial B}^0,\partial B)$ we first remark that $\mathbf{w}_0 := \mathbf{W}_{\epsilon_0} \mathbf{h}$ lies in the space
\begin{equation*}
\mathbf{Y}^0(B) := \{\mathbf{u}\in\mathbf{L}^2(B) \mid \curl\mathbf{u}\in\mathbf{L}^2(B), \; \div(\epsilon_0\mathbf{u}) = 0, \; \un\cdot(\epsilon_0\mathbf{u}) = 0\},
\end{equation*}
equipped with the standard norm on $\Hcurl{B}$, where the last condition in the space follows from the observation that
\begin{equation*}
0 = \sdiv_{\partial B}(\un\times\curl\mathbf{w}_0) = -\un\cdot\curl\curl\mathbf{w}_0 = -k^2\un\cdot(\epsilon_0\mathbf{w}_0).
\end{equation*}
Since $\epsilon|_D\in W_\Sigma^{1,\infty}(D)$ and $\epsilon=1$ in $B\setminus\overline{D}$, we may apply the regularity results in \cite[Proposition 6.5]{ciarlet} (see also \cite{bonito_guermond_luddens}) to assert the existence of $\tau_{\epsilon_0}\in(0,\frac{1}{2})$ such that the space $\mathbf{Y}^0(B)$ is continuously embedded into $\mathbf{H}^s(B)$ for all $s\in[0,\tau_{\epsilon_0})$, and we fix a positive such value of $s$ for the remainder of our discussion. We conclude that $\mathbf{w}_0\in\mathbf{H}^s(B)$ satisfies the estimate
\begin{equation*}
\norm{\mathbf{w}_0}_{\mathbf{H}^s(B)} \le C_s\norm{\mathbf{w}_0}_{\Hcurl{B}} \le C_{s,\epsilon_0} \norm{\mathbf{h}}_{\mathbf{H}(\sdiv_{\partial B}^0,\partial B)}.
\end{equation*}
Furthermore, the Sobolev embedding theorem (cf. \cite{adams}) implies that $\mathbf{w}_0\in\mathbf{L}^{3/(3-2s)}(B)$ with continuous embedding, and we use this $L^p$-regularity to establish \eqref{MW_est}. We see from H{\"o}lder's inequality that
\begin{align*}
\norm{\mathbf{M}_{\epsilon_1,\epsilon_0} \mathbf{W}_{\epsilon_0} \mathbf{h}}_{\mathbf{L}^2(B)} &= \left( \int_B \abs{\epsilon_1 - \epsilon_0}^2 \abs{\mathbf{w}_0}^2 \,dx \right)^{1/2} \\
&\le \norm{\epsilon_1 - \epsilon_0}_{L^{3/s}(B)} \norm{\mathbf{w}_0}_{\mathbf{L}^{3/(3-2s)}(B)} \\
&\le C_s \norm{\epsilon_1 - \epsilon_0}_{L^{3/s}(B)} \norm{\mathbf{w}_0}_{\mathbf{H}^s(B)} \\
&\le C_{s,\epsilon_0} \norm{\epsilon_1 - \epsilon_0}_{L^{3/s}(B)} \norm{\mathbf{h}}_{\mathbf{H}(\sdiv_{\partial B}^0,\partial B)},
\end{align*}
and the definition of the operator norm now results in \eqref{MW_est}.
\end{proof}

Examining the initial estimate \eqref{psi_est1}, we see that it suffices to show that $\norm{\mathbf{V}_{\epsilon_1}}$ is small whenever $\epsilon_1$ is in a small neighborhood of $\epsilon_0$. Ideally, we would be able to measure this neighborhood in the same norm that appeared in the result of Lemma \ref{lemma:MW_est}; however, since the domain of $\mathbf{V}_{\epsilon_1}$ is merely $\mathbf{L}^2(B)$, we are unable to leverage the same regularity results from \cite{bonito_guermond_luddens,ciarlet} that gave rise to the $L^{3/s}(B)$-norm. Moreover, defining the domain of this operator to be a smaller space would require the same redefinition for the codomain of the operator $\mathbf{M}_{\epsilon_1,\epsilon_0}$, and the result of Lemma \ref{lemma:MW_est} would be invalidated. Thus, in the following lemma we remain content to measure the perturbation $\epsilon_1 - \epsilon_0$ in $L^\infty(B)$. Unlike the problem considered in \cite{cogar2020} for the Helmholtz equation, we remark that this present difficulty is caused by the strong dependence of regularity results for Maxwell's equations on the coefficients.

\begin{lemma} \label{lemma:V_bound}

The operator norm $\norm{\mathbf{V}_{\epsilon_1}}$ is uniformly bounded whenever $\norm{\epsilon_1-\epsilon_0}_{L^\infty(B)}$ is sufficiently small.

\end{lemma}

\begin{proof}

We derive a factorization of $\mathbf{V}_{\epsilon_1}$ that is valid when $\norm{\epsilon_1 - \epsilon_0}_{L^\infty(B)}$ is small. For a given $\mathbf{f}\in\mathbf{L}^2(B)$ we consider $\mathbf{v}_j := \mathbf{V}_{\epsilon_j} \mathbf{f}$, $j = 0,1$, and we see that $\mathbf{v}_1 - \mathbf{v}_0$ satisfies
\begin{align*}
\curl\curl(\mathbf{v}_1 - \mathbf{v}_0) - k^2 \epsilon_1 (\mathbf{v}_1 - \mathbf{v}_0) &= k^2(\epsilon_1-\epsilon_0)\mathbf{v}_0 \text{ in } B, \\
\un\times\curl(\mathbf{v}_1 - \mathbf{v}_0) - z\mathcal{S}_\delta (\mathbf{v}_{1,T} - \mathbf{v}_{0,T}) &= \mathbf{0} \text{ on } \partial B.
\end{align*}
It follows that $\mathbf{v}_1 - \mathbf{v}_0 = \mathbf{V}_{\epsilon_1}\biggr[ k^2(\epsilon_1 - \epsilon_0) \mathbf{v}_0 \biggr]$, and we arrive at the factorization
\begin{equation*}
\mathbf{V}_{\epsilon_1} - \mathbf{V}_{\epsilon_0} = k^2 \mathbf{V}_{\epsilon_1} \mathbf{M}_{\epsilon_1,\epsilon_0} \mathbf{V}_{\epsilon_0},
\end{equation*}
which may be written as
\begin{equation*}
\mathbf{V}_{\epsilon_1} \biggr( I - k^2\mathbf{M}_{\epsilon_1,\epsilon_0} \mathbf{V}_{\epsilon_0} \biggr) = \mathbf{V}_{\epsilon_0}.
\end{equation*}
From the observation that for all $\mathbf{f}\in\mathbf{L}^2(B)$ we have
\begin{align*}
\norm{\mathbf{M}_{\epsilon_1,\epsilon_0} \mathbf{V}_{\epsilon_0} \mathbf{f}}_{\mathbf{L}^2(B)} &= \left( \int_B \abs{\epsilon_1-\epsilon_0}^2 \abs{\mathbf{v}_0}^2 \,dx \right)^{1/2} \\
&\le \norm{\epsilon_1-\epsilon_0}_{L^\infty(B)} \norm{\mathbf{v}_0}_{\mathbf{L}^2(B)} \\
&\le \norm{\mathbf{V}_{\epsilon_0}} \norm{\epsilon_1-\epsilon_0}_{L^\infty(B)} \norm{\mathbf{f}}_{\mathbf{L}^2(B)},
\end{align*}
we have the estimate $\norm{\mathbf{M}_{\epsilon_1,\epsilon_0} \mathbf{V}_{\epsilon_0}} \le \norm{\mathbf{V}_{\epsilon_0}} \norm{\epsilon_1-\epsilon_0}_{L^\infty(B)}$, and it follows that the operator $I-k^2\mathbf{M}_{\epsilon_1,\epsilon_0}\mathbf{V}_{\epsilon_0}:\mathbf{L}^2(B)\to\mathbf{L}^2(B)$ is invertible whenever $\norm{\epsilon_1-\epsilon_0}_{L^\infty(B)} < k^{-2}\norm{\mathbf{V}_{\epsilon_0}}^{-1}$. Thus, in this case we have
\begin{equation*}
\mathbf{V}_{\epsilon_1} = \mathbf{V}_{\epsilon_0} \biggr( I - k^2\mathbf{M}_{\epsilon_1,\epsilon_0} k^{-2}\mathbf{V}_{\epsilon_0} \biggr)^{-1},
\end{equation*}
and from a Neumann series expansion we obtain
\begin{equation*}
\norm{\mathbf{V}_{\epsilon_1}} \le \frac{\norm{\mathbf{V}_{\epsilon_0}}}{1 - k^2 \norm{\mathbf{V}_{\epsilon_0}}\norm{\epsilon_1 - \epsilon_0}_{L^\infty(B)}}.
\end{equation*}
This inequality allows us to conclude that $\norm{\mathbf{V}_{\epsilon_1}}$ is uniformly bounded whenever $\norm{\epsilon_1-\epsilon_0}_{L^\infty(B)} < k^{-2}\norm{\mathbf{V}_{\epsilon_0}}^{-1}$.
\end{proof}

Combining \eqref{psi_est1} with the results of Lemmas \ref{lemma:MW_est}--\ref{lemma:V_bound} yields our main result, which we state in the following theorem.

\begin{theorem} \label{theorem:psi_est}

If $\epsilon_0$ is fixed and $\norm{\epsilon_1 - \epsilon_0}_{L^\infty(B)}$ is sufficiently small, then there exist constants $s\in(0,\frac{1}{2})$ and $C_{s,\epsilon_0}>0$ independent of $\epsilon_1$ for which
\begin{equation} \label{psi_est2}
\norm{\boldPsi_{z,\epsilon_1}^{(\delta)} - \boldPsi_{z,\epsilon_0}^{(\delta)}} \le C_{s,\epsilon_0} \norm{\epsilon_1-\epsilon_0}_{L^{3/s}(B)}.
\end{equation}

\end{theorem}

\begin{remark}

We remark that, assuming a suitably small magnitude of the perturbation, i.e. small $\norm{\epsilon_1-\epsilon_0}_{L^\infty(B)}$, Theorem \ref{theorem:psi_est} implies that the difference in the corresponding solution operators is stable with respect to the measure of the perturbed region in $B$ through the norm $\norm{\epsilon_1-\epsilon_0}_{L^{3/s}(B)}$. In particular, if $\Omega_{\epsilon_1,\epsilon_0} = \supp(\epsilon_1-\epsilon_0)$, then we have
\begin{equation*}
\norm{\epsilon_1-\epsilon_0}_{L^{3/s}(B)} \le \abs{\Omega_{\epsilon_1,\epsilon_0}}^{s/3} \norm{\epsilon_1-\epsilon_0}_{L^\infty(B)}.
\end{equation*}
In the special case that $\Omega_{\epsilon_1,\epsilon_0}$ is a smooth deformation of a ball of radius $r(\epsilon_1)$ dependent upon $\epsilon_1$, the estimate \eqref{psi_est2} implies that
\begin{equation*}
\norm{\boldPsi_{z,\epsilon_1}^{(\delta)} - \boldPsi_{z,\epsilon_0}^{(\delta)}} \le C_{s,\epsilon_0} r(\epsilon_1)^s \norm{\epsilon_1-\epsilon_0}_{L^\infty(B)}.
\end{equation*}
We note that the convergence rate with respect to the measure of the perturbed region is dependent upon $s$, which is, roughly speaking, representative of the regularity of the interface $\Sigma$ associated with the space $W_\Sigma^{1,\infty}(D)$ in which $\epsilon_0|_D$ lies. The same results applied to a permittivity $\epsilon_0$ which is in $C^1(B)$ would yield $s = \frac{1}{2}$ (cf. \cite{bonito_guermond_luddens}).

\end{remark}

The result of Theorem \ref{theorem:psi_est} immediately implies the following corollary (cf. \cite{kato}).

\begin{corollary} \label{corollary:eigs_epsilon}

If $\norm{\epsilon_1-\epsilon_0}_{L^\infty(B)}$ is sufficiently small and $\epsilon_1\to\epsilon_0$ in $L^{3/s}(B)$, then the set of electromagnetic $\delta$-Stekloff eigenvalues for $\epsilon_1$ converges to the set of electromagnetic $\delta$-Stekloff eigenvalues for $\epsilon_0$.

\end{corollary}


\subsection{Stability with respect to $\delta$} \label{subsec_pert_delta}

In this section we fix the permittivity $\epsilon$, and we show that the eigenvalues are stable with respect to changes in $\delta$. Since the smoothing parameter $\delta$ may be freely chosen, our main interest is in showing stability at the point $\delta=0$, which implies that the electromagnetic $\delta$-Stekloff eigenvalues converge to the standard electromagnetic Stekloff eigenvalues as $\delta\to0^+$. We accomplish this task using a factorization technique similar to Section \ref{subsec_pert_epsilon}. Before we begin, we remark that for technical reasons we will consider a slightly different solution operator $\tilde{\boldPsi}_z^{(\delta)}:\mathbf{H}(\sdiv_{\partial B}^0,\partial B)\to\mathbf{H}(\sdiv_{\partial B}^0,\partial B)$ defined by $\tilde{\boldPsi}_z^{(\delta)} := \mathcal{S}_\delta \mathbf{T}_z^{(\delta)}$. By applying Proposition \ref{prop:Psi} with $\delta_1=1$ and $\delta_2=0$, we observe that the spectrum of $\tilde{\boldPsi}_z^{(\delta)}$ coincides with that of $\boldPsi_z^{(\delta)}$, and as a consequence we may equivalently study the spectrum of $\tilde{\boldPsi}_z^{(\delta)}$. We will often measure norms of operators between spaces $\mathbf{H}^{\rho_1}(\sdiv_{\partial B}^0,\partial B)$ and $\mathbf{H}^{\rho_2}(\sdiv_{\partial B}^0,\partial B)$ for some $\rho_1,\rho_2\ge0$, and for convenience we denote this operator norm as $\norm{\cdot}_{\rho_1,\rho_2}$. \par

We now derive a factorization for the operator $\mathbf{T}_z^{(\delta)}$. For each $\delta\ge0$ we denote by $\mathbf{u}_\delta\in\Hcurl{B}$ the unique solution of \eqref{defT} for a given $\mathbf{h}\in\mathbf{H}(\sdiv_{\partial B}^0,\partial B)$, which implies that $\mathcal{S}_0 \mathbf{u}_{\delta,T} = \mathbf{T}_z^{(\delta)} \mathbf{h}$. We see that $\mathbf{u}_\delta - \mathbf{u}_0$ satisfies
\begin{align*}
\curl\curl(\mathbf{u}_\delta - \mathbf{u}_0) - k^2 \epsilon (\mathbf{u}_\delta - \mathbf{u}_0) &= \mathbf{0} \text{ in } B, \\
\un\times\curl(\mathbf{u}_\delta - \mathbf{u}_0) - z\mathcal{S}_\delta (\mathbf{u}_{\delta,T} - \mathbf{u}_{0,T}) &= z(\mathcal{S}_\delta - \mathcal{S}_0)\mathbf{u}_{0,T} \text{ on } \partial B,
\end{align*}
from which we obtain
\begin{equation*}
(\mathbf{T}_z^{(\delta)} - \mathbf{T}_z^{(0)})\mathbf{h} = \mathcal{S}_0 \left(\mathbf{u}_{\delta,T} - \mathbf{u}_{0,T}\right) = \mathbf{T}_z^{(\delta)} \biggr[ z(\mathcal{S}_\delta - \mathcal{S}_0)\mathbf{u}_{0,T} \biggr] = z\mathbf{T}_z^{(\delta)} (\mathcal{S}_\delta - \mathcal{S}_0) \mathbf{T}_z^{(0)} \mathbf{h}.
\end{equation*}
We note that the final equality follows from the fact that $(\mathcal{S}_\delta - \mathcal{S}_0) = (\mathcal{S}_\delta - \mathcal{S}_0) \mathcal{S}_0$. Since this equation holds for all $\mathbf{h}$, we arrive at the factorization
\begin{equation} \label{delta_factor}
\mathbf{T}_z^{(\delta)} - \mathbf{T}_z^{(0)} = z\mathbf{T}_z^{(\delta)} (\mathcal{S}_\delta - \mathcal{S}_0)\mathbf{T}_z^{(0)},
\end{equation}
which may be written as
\begin{equation} \label{delta_factor2}
\mathbf{T}_z^{(\delta)} \biggr[ I - z(\mathcal{S}_\delta - \mathcal{S}_0)\mathbf{T}_z^{(0)} \biggr] = \mathbf{T}_z^{(0)}.
\end{equation}
In order to invert the operator $I - z(\mathcal{S}_\delta - \mathcal{S}_0)\mathbf{T}_z^{(0)}:\mathbf{H}(\sdiv_{\partial B}^0,\partial B)\to\mathbf{H}(\sdiv_{\partial B}^0,\partial B)$ for small $\delta$, we show in the following lemma that $\mathcal{S}_\delta \to \mathcal{S}_0$ in a certain operator norm as $\delta\to0^+$.

\begin{lemma} \label{lemma:S_delta}

The operator $\mathcal{S}_\delta:\mathbf{H}^1(\div_{\partial B}^0,\partial B)\to\mathbf{H}(\sdiv_{\partial B}^0,\partial B)$ converges in operator norm to $\mathcal{S}_0:\mathbf{H}^1(\div_{\partial B}^0,\partial B)\to\mathbf{H}(\sdiv_{\partial B}^0,\partial B)$ as $\delta\to0^+$.

\end{lemma}

\begin{proof}

For a given $\mathbf{h}\in\mathbf{H}^1(\sdiv_{\partial B}^0,\partial B)$ with eigenfunction expansion
\begin{equation*}
\mathbf{h} = \sum_{m=1}^\infty \mathbf{h}_m^{(2)} \veccurl_{\partial B} Y_m,
\end{equation*}
we observe that
\begin{align}
\begin{split} \label{Sdelta_est}
\norm{(\mathcal{S}_\delta - \mathcal{S}_0)\mathbf{h}}_{\mathbf{H}(\sdiv_{\partial B}^0,\partial B)} &= \norm{\sum_{m=1}^\infty (\mu_m^{-\delta} - 1) \mathbf{h}_m^{(2)} \veccurl_{\partial B} Y_m}_{\mathbf{L}_t^2(\partial B)} \\
&= \left( \sum_{m=1}^\infty \mu_m \abs{ (\mu_m^{-\delta} - 1) \mathbf{h}_m^{(2)} }^2 \right)^{1/2} \\
&= \left( \sum_{m=1}^\infty \abs{\frac{1-\mu_m^{-\delta}}{\mu_m^{1/2}}}^2 \mu_m^2 \abs{\mathbf{h}_m^{(2)}}^2 \right)^{1/2}.
\end{split}
\end{align}
For each $m\in\mathbb{N}$ it is clear that
\begin{equation*}
\lim_{\delta\to0^+} \frac{1-\mu_m^{-\delta}}{\mu_m^{1/2}} = 0,
\end{equation*}
but the convergence may not be uniform in $m$. However, we now show that there exists some $m_*\in\mathbb{N}$ independent of $\delta$ for which the sequence $\left\{\frac{1-\mu_m^{-\delta}}{\mu_m^{1/2}}\right\}_{m\ge m_*}$ is non-increasing, which will allow us to conclude the desired result. For each $\delta>0$ we consider the function $\varphi_\delta(t) := \frac{1-t^{-\delta}}{t^{1/2}}$ on the interval $[1,\infty)$. We see that
\begin{equation*}
\varphi_\delta'(e^2) = \frac{1}{2}e^{-2\delta-3} ( 2\delta + 1 - e^{2\delta} ),
\end{equation*}
and since $2\delta+1 - e^{2\delta} < 0$ for all $\delta>0$ we conclude that $\varphi_\delta$ is decreasing on the interval $(e^2,\infty)$ for all $\delta>0$. Thus, by choosing $m_*$ such that $\mu_{m_*} > e^2$ we have a non-increasing sequence $\left\{\frac{1-\mu_m^{-\delta}}{\mu_m^{1/2}}\right\}_{m\ge m_*}$, and we may split the final series in \eqref{Sdelta_est} to obtain
\begin{equation*}
\norm{(\mathcal{S}_\delta - \mathcal{S}_0)\mathbf{h}}_{\mathbf{H}(\sdiv_{\partial B}^0,\partial B)} \le \max_{1\le m\le m_*} \abs{\frac{1-\mu_m^{-\delta}}{\mu_m^{1/2}}} \norm{\mathbf{h}}_{\mathbf{H}^1(\sdiv_{\partial B}^0,\partial B)}.
\end{equation*}
This result leads us to the estimate
\begin{equation*}
\norm{\mathcal{S}_\delta - \mathcal{S}_0}_{1,0} \le \max_{1\le m\le m_*} \abs{\frac{1-\mu_m^{-\delta}}{\mu_m^{1/2}}},
\end{equation*}
and by applying the Mean Value Theorem to each of the functions $\psi_m(t) := 1-\mu_m^{-t}$ on the interval $[0,\delta]$ we obtain
\begin{equation} \label{S_delta_est}
\norm{\mathcal{S}_\delta - \mathcal{S}_0}_{1,0} \le C\delta,
\end{equation}
where the constant $C$ is independent of $\delta$. We note that this constant is related to the maximum in the previous inequality, which is over a finite collection of terms and hence exists. This final estimate provides the desired result.
\end{proof}

We return to \eqref{delta_factor2}, and by the result of Lemma \ref{lemma:S_delta} we may invert the operator in brackets for sufficiently small $\delta>0$ in order to arrive at the representation
\begin{equation*}
\mathbf{T}_z^{(\delta)} = \mathbf{T}_z^{(0)} \biggr[ I - z(\mathcal{S}_\delta - \mathcal{S}_0)\mathbf{T}_z^{(0)} \biggr]^{-1}.
\end{equation*}
From a Neumann series expansion and \eqref{S_delta_est} we have
\begin{equation*}
\norm{\mathbf{T}_z^{(\delta)}}_{0,1} \le \frac{\norm{\mathbf{T}_z^{(0)}}_{0,1}}{1 - C_0\delta}
\end{equation*}
for sufficiently small $\delta>0$, where $C_0$ is a constant independent of $\delta$, and it follows that $\norm{\mathbf{T}_z^{(\delta)}}_{0,1}$ is uniformly bounded for small $\delta>0$. With this result in hand we may use \eqref{delta_factor} and Lemma \ref{lemma:S_delta} to immediately obtain the estimate
\begin{equation} \label{T_delta}
\norm{\mathbf{T}_z^{(\delta)} - \mathbf{T}_z^{(0)}}_{0,1} \le C \delta
\end{equation}
for sufficiently small $\delta>0$. It now remains to establish the same estimate for the difference $\tilde{\boldPsi}_z^{(\delta)} - \tilde{\boldPsi}_z^{(0)}$, which we give in the following theorem.

\begin{theorem} \label{theorem:Psi_delta}

If $\delta>0$ is sufficiently small, then there exists a constant $C$ independent of $\delta$ for which
\begin{equation} \label{Psi_delta}
\norm{\tilde{\boldPsi}_z^{(\delta)} - \tilde{\boldPsi}_z^{(0)}}_{0,0} \le C \delta.
\end{equation}

\end{theorem}

\begin{proof}

By the triangle inequality we have
\begin{align*}
\norm{\tilde{\boldPsi}_z^{(\delta)} - \tilde{\boldPsi}_z^{(0)}}_{0,0} &= \norm{\mathcal{S}_\delta \mathbf{T}_z^{(\delta)} - \mathcal{S}_0 \mathbf{T}_z^{(0)}}_{0,0} \\
&\le \norm{\mathcal{S}_\delta ( \mathbf{T}_z^{(\delta)} - \mathbf{T}_z^{(0)} )}_{0,0} + \norm{(\mathcal{S}_\delta - \mathcal{S}_0) \mathbf{T}_z^{(0)}}_{0,0} \\
&\le \norm{\mathcal{S}_\delta}_{1,0} \norm{\mathbf{T}_z^{(\delta)} - \mathbf{T}_z^{(0)}}_{0,1}  + \norm{\mathcal{S}_\delta - \mathcal{S}_0}_{1,0} \norm{\mathbf{T}_z^{(0)}}_{0,1}.
\end{align*}
Noting that $\norm{\mathcal{S}_\delta}_{1,0}$ is uniformly bounded for small $\delta$ as a result of Lemma \ref{lemma:S_delta}, the desired estimate \eqref{Psi_delta} follows from combining \eqref{S_delta_est} and \eqref{T_delta}.
\end{proof}

The norm convergence of the sequence $\left\{\tilde{\boldPsi}_z^{(\delta)}\right\}_{\delta>0}$ that we obtained in Theorem \ref{theorem:Psi_delta} immediately implies the following corollary (cf. \cite{kato}).

\begin{corollary} \label{corollary:eigs_delta}

The electromagnetic $\delta$-Stekloff eigenvalues converge to the standard electromagnetic Stekloff eigenvalues as $\delta\to0^+$.

\end{corollary}

%
%

\section{Conclusion}

\label{sec_conclusion}

We have shown that, with a slight modification of the boundary condition in the standard electromagnetic Stekloff eigenvalue problem, infinitely many eigenvalues exist even for an absorbing medium, and these eigenvalues are stable with respect to changes in the material coefficients and the smoothing parameter $\delta$. Both of these results are useful in establishing applicability and robustness of nondestructive evaluation methods based on using eigenvalues as potential target signatures, and it might be useful to apply the same ideas to other types of problems besides Stekloff eigenvalues and their generalizations. In particular, the recent introduction in \cite{cogar_monk} of a class of eigenvalues that depends on a tuning parameter $\gamma$ may allow for some control over the sensitivity of eigenvalues to changes in the medium, but existence results for an absorbing medium are lacking. However, a similar introduction of a smoothing operator into this problem is not straightforward, as the eigenparameter no longer appears in the boundary condition. The effort to find the proper trace class modification of this problem is ongoing.


\begin{thebibliography}{10}

\bibitem{adams}
\newblock R.~Adams,
\newblock \emph{Sobolev {S}paces},
\newblock Academic Press, New York-London, 1975,
\newblock Pure and Applied Mathematics, Vol. 65.

\bibitem{audibert_cakoni_haddar}
\newblock L.~Audibert, F.~Cakoni and H.~Haddar,
\newblock New sets of eigenvalues in inverse scattering for inhomogeneous media
  and their determination from scattering data,
\newblock \emph{Inverse Problems}, \textbf{33} (2017), 125011,
\newblock \urlprefix\url{http://stacks.iop.org/0266-5611/33/i=12/a=125011}.

\bibitem{audibert_chesnel_haddar2018}
\newblock L.~Audibert, L.~Chesnel and H.~Haddar,
\newblock Transmission eigenvalues with artificial background for explicit
  material index identification,
\newblock \emph{C. R. Math. Acad. Sci. Paris}, \textbf{356} (2018), 626--631,
\newblock \urlprefix\url{https://doi.org/10.1016/j.crma.2018.04.015}.

\bibitem{audibert_chesnel_haddar2019}
\newblock L.~Audibert, L.~Chesnel and H.~Haddar,
\newblock Inside-outside duality with artificial backgrounds,
\newblock \emph{Inverse Problems}, \textbf{35} (2019), 104008, 26,
\newblock \urlprefix\url{https://doi.org/10.1088/1361-6420/ab3244}.

\bibitem{bi_zhang_yang}
\newblock H.~Bi, Y.~Zhang and Y.~Yang,
\newblock Two-grid discretizations and a local finite element scheme for a
  non-selfadjoint {S}tekloff eigenvalue problem,
\newblock \emph{Comput. Math. Appl.}, \textbf{79} (2020), 1895--1913,
\newblock
  \urlprefix\url{https://www.sciencedirect.com/science/article/pii/S0898122118304735?via\%3Dihub}.

\bibitem{bonito_guermond_luddens}
\newblock A.~Bonito, J.~Guermond and F.~Luddens,
\newblock Regularity of the {M}axwell equations in heterogeneous media and
  {L}ipschitz domains,
\newblock \emph{J. Math. Anal. Appl.}, \textbf{408} (2013), 498--512,
\newblock \urlprefix\url{https://doi.org/10.1016/j.jmaa.2013.06.018}.

\bibitem{cakoni_colton_haddar}
\newblock F.~Cakoni, D.~Colton and H.~Haddar,
\newblock \emph{Inverse scattering theory and transmission eigenvalues},
  vol.~88 of CBMS-NSF Regional Conference Series in Applied Mathematics,
\newblock Society for Industrial and Applied Mathematics (SIAM), Philadelphia,
  PA, 2016,
\newblock \urlprefix\url{https://epubs.siam.org/doi/10.1137/1.9781611974461}.

\bibitem{cakoni_colton_meng_monk}
\newblock F.~Cakoni, D.~Colton, S.~Meng and P.~Monk,
\newblock Stekloff eigenvalues in inverse scattering,
\newblock \emph{SIAM J. Appl. Math.}, \textbf{76} (2016), 1737--1763,
\newblock \urlprefix\url{https://epubs.siam.org/doi/10.1137/16M1058704}.

\bibitem{camano_lackner_monk}
\newblock J.~Cama{\~{n}}o, C.~Lackner and P.~Monk,
\newblock Electromagnetic {S}tekloff eigenvalues in inverse scattering,
\newblock \emph{SIAM J. Math. Anal.}, \textbf{49} (2017), 4376--4401,
\newblock \urlprefix\url{https://doi.org/10.1137/16M1108893}.

\bibitem{ciarlet}
\newblock P.~Ciarlet,
\newblock On the approximation of electromagnetic fields by edge finite
  elements. part 3: sensitivity to coefficients,
\newblock \emph{hal-02276430},
\newblock \urlprefix\url{https://epubs.siam.org/doi/10.1137/19M1275383}.

\bibitem{cogar}
\newblock S.~Cogar,
\newblock A modified transmission eigenvalue problem for scattering by a
  partially coated crack,
\newblock \emph{Inverse Problems}, \textbf{34} (2018), 115003, 29,
\newblock \urlprefix\url{https://doi.org/10.1088/1361-6420/aadb20}.

\bibitem{cogar2019}
\newblock S.~Cogar,
\newblock \emph{New Eigenvalue Problems in Inverse Scattering},
\newblock PhD thesis, University of Delaware, 2019,
\newblock
  \urlprefix\url{https://search.proquest.com/docview/2268338078?accountid=13626}.

\bibitem{cogar2020}
\newblock S.~Cogar,
\newblock Analysis of a trace class {S}tekloff eigenvalue problem arising in
  inverse scattering,
\newblock \emph{SIAM J. Appl. Math.}, \textbf{80} (2020), 881--905,
\newblock \urlprefix\url{https://epubs.siam.org/doi/10.1137/19M1295155}.

\bibitem{cogar_colton_meng_monk}
\newblock S.~Cogar, D.~Colton, S.~Meng and P.~Monk,
\newblock Modified transmission eigenvalues in inverse scattering theory,
\newblock \emph{Inverse Problems}, \textbf{33} (2017), 125002,
\newblock \urlprefix\url{http://stacks.iop.org/0266-5611/33/i=12/a=125002}.

\bibitem{cogar_colton_monk}
\newblock S.~Cogar, D.~Colton and P.~Monk,
\newblock Using eigenvalues to detect anomalies in the exterior of a cavity,
\newblock \emph{Inverse Problems}, \textbf{34} (2018), 085006, 27,
\newblock \urlprefix\url{https://doi.org/10.1088/1361-6420/aac8ef}.

\bibitem{cogar_colton_monk2019}
\newblock S.~Cogar, D.~Colton and P.~Monk,
\newblock Eigenvalue problems in inverse electromagnetic scattering theory,
\newblock in \emph{Maxwell's Equations: Analysis and Numerics} (eds. U.~Langer,
  D.~Pauly and S.~Repin),
\newblock Radon Series on Computational and Applied Mathematics, De Gruyter,
  2019,
\newblock chapter~5,
\newblock \urlprefix\url{https://arxiv.org/abs/1805.06986}.

\bibitem{cogar_monk}
\newblock S.~Cogar and P.~Monk,
\newblock Modified electromagnetic transmission eigenvalues in inverse
  scattering theory,
\newblock \urlprefix\url{https://arxiv.org/abs/2005.14277}.

\bibitem{colton_kress}
\newblock D.~Colton and R.~Kress,
\newblock \emph{Inverse acoustic and electromagnetic scattering theory},
  vol.~93 of Applied Mathematical Sciences,
\newblock 4th edition,
\newblock Springer, Cham, 2019,
\newblock
  \urlprefix\url{https://link.springer.com/book/10.1007\%2F978-3-030-30351-8}.

\bibitem{costabel}
\newblock M.~Costabel,
\newblock A remark on the regularity of solutions of {M}axwell's equations on
  {L}ipschitz domains,
\newblock \emph{Math. Methods Appl. Sci.}, \textbf{12} (1990), 365--368,
\newblock \urlprefix\url{https://doi.org/10.1002/mma.1670120406}.

\bibitem{colwell}
\newblock F.~Gylys-Colwell,
\newblock An inverse problem for the {H}elmholtz equation,
\newblock \emph{Inverse Problems}, \textbf{12} (1996), 139--156,
\newblock
  \urlprefix\url{https://iopscience.iop.org/article/10.1088/0266-5611/12/2/003}.

\bibitem{halla2}
\newblock M.~Halla,
\newblock Electromagnetic {S}tekloff eigenvalues: approximation analysis,
\newblock \urlprefix\url{https://arxiv.org/abs/1909.00689}.

\bibitem{halla1}
\newblock M.~Halla,
\newblock Electromagnetic {S}tekloff eigenvalues: existence and behavior in the
  selfadjoint case,
\newblock \urlprefix\url{https://arxiv.org/abs/1909.01983}.

\bibitem{jost}
\newblock J.~Jost,
\newblock \emph{Riemannian Geometry and Geometric Analysis},
\newblock Seventh edition,
\newblock Universitext, Springer, Cham, 2017,
\newblock \urlprefix\url{https://doi.org/10.1007/978-3-319-61860-9}.

\bibitem{kato}
\newblock T.~Kato,
\newblock \emph{Perturbation Theory for Linear Operators},
\newblock Classics in Mathematics, Springer-Verlag, Berlin, 1995,
\newblock Reprint of the 1980 edition.

\bibitem{kirsch_hettlich}
\newblock A.~Kirsch and F.~Hettlich,
\newblock \emph{The {M}athematical {T}heory of {T}ime-{H}armonic {M}axwell's
  {E}quations}, vol. 190 of Applied Mathematical Sciences,
\newblock Springer, Cham, 2015,
\newblock \urlprefix\url{https://doi.org/10.1007/978-3-319-11086-8}.

\bibitem{liu_liu_sun}
\newblock J.~Liu, Y.~Liu and J.~Sun,
\newblock An inverse medium problem using {S}tekloff eigenvalues and a
  {B}ayesian approach,
\newblock \emph{Inverse Problems}, \textbf{35} (2019), 094004, 20,
\newblock
  \urlprefix\url{https://iopscience.iop.org/article/10.1088/1361-6420/ab1be9}.

\bibitem{monk}
\newblock P.~Monk,
\newblock \emph{Finite Element Methods for {M}axwell's Equations},
\newblock Numerical Mathematics and Scientific Computation, Oxford University
  Press, New York, 2003,
\newblock
  \urlprefix\url{https://doi.org/10.1093/acprof:oso/9780198508885.001.0001}.

\bibitem{nedelec}
\newblock J.-C. N\'{e}d\'{e}lec,
\newblock \emph{Acoustic and Electromagnetic Equations}, vol. 144 of Applied
  Mathematical Sciences,
\newblock Springer-Verlag, New York, 2001,
\newblock \urlprefix\url{https://doi.org/10.1007/978-1-4757-4393-7}.

\bibitem{ringrose}
\newblock J.~R. Ringrose,
\newblock \emph{Compact Non-Self-Adjoint Operators},
\newblock Van Nostrand Reinhold Co., London, 1971.

\bibitem{sayas_brown_hassell}
\newblock F.~Sayas, T.~Brown and M.~Hassell,
\newblock \emph{Variational Techniques for Elliptic Partial Differential
  Equations},
\newblock CRC Press, 2019,
\newblock \urlprefix\url{https://www.taylorfrancis.com/books/9780429507069}.

\end{thebibliography}

\providecommand{\href}[2]{#2}
\providecommand{\arxiv}[1]{\href{http://arxiv.org/abs/#1}{arXiv:#1}}
\providecommand{\url}[1]{\texttt{#1}}
\providecommand{\urlprefix}{URL }

\end{document}